\documentclass [11pt] {article}
\usepackage{amsfonts}
\usepackage{amssymb}
\usepackage{times}
\usepackage{color}
\usepackage{graphicx}

\newcommand{\Hom}{\mathrm{Hom}}

\newtheorem{lemma}{Lemma}[section]
\newtheorem{corollary}{Corollary}[section]
\newtheorem{theorem}{Theorem}

\newtheorem{remark}{Remark}[section]

\newtheorem{proposition}{Proposition}[section]
\newtheorem{definition}{Definition}[section]

\def\Hom{\mbox{Hom}\,}

\def\<{\langle}
\def\>{\rangle}

\def\to{\rightarrow}

\def\bX{{\bf X}}

\def\bA{{\bf A}}

\begin{document}

\title{Gowers norms, regularization and limits of functions on abelian groups}
\author{{\sc Bal\'azs Szegedy}}

\maketitle

\abstract{For every natural number $k$ we prove a decomposition theorem for bounded measurable functions on compact abelian groups into a structured part, a quasi random part and a small error term. In this theorem quasi randomness is measured with the Gowers norm $U_{k+1}$ and the structured part is a bounded complexity ``nilspace-polynomial'' of degree $k$. This statement implies a general inverse theorem for the $U_{k+1}$ norm. (We discuss some consequences in special families of groups such as bounded exponent groups, zero characteristic groups and the circle group.) Along these lines we introduce a convergence notion and corresponding limit objects for functions on abelian groups. This subject is closely related to the recently developed graph and hypergraph limit theory. An important goal of this paper is to put forward a new algebraic aspect of the notion ``higher order Fourier analysis''. According to this, $k$-th order Fourier analysis is regarded as the study of continuous morphisms between structures called compact $k$-step nilspaces. All our proofs are based on an underlying theory of topological nilspace factors of ultra product groups.}

\tableofcontents

\section{Introduction}

Higher order Fourier analysis is a notion which has many aspects and interpretations. The subject originates in a fundamental work by Gowers \cite{Gow},\cite{Gow2} in which he introduced a sequence of norms for functions on abelian groups and he used them to prove quantitative bounds for Szemer\'edi's theorem on arithmetic progressions \cite{Szem1} Since then many results were published towards a better understanding of the Gowers norms \cite{GrTao},\cite{GrTao2},\cite{GTZ},\cite{GowW},\cite{GowW2},\cite{GowW3},\cite{TZ},\cite{Sz1},\cite{Sz2},\cite{Sz3}
Common themes in all these works are the following four topics:

\medskip

\noindent{1.)~}{\it Inverse theorems}

\noindent{2.)~}{\it Decompositions of functions into structured and random parts}

\noindent{3.)~}{\it Counting structures in subsets and functions on abelian groups}

\noindent{4.)~}{\it Connection to ergodic theory and nilmanifolds}

\medskip

In the present paper we wish to contribute to all of these topics however we also put forward three other directions:

\medskip

\noindent{5.)~}{\it An algebraic language based on morphisms between nilspaces}

\noindent{6.)~}{\it The compact case}

\noindent{7.)~}{\it Limit objects for structures in abelian groups}

Important results on the fifth and sixth topics were also obtained by Host and Kra in the papers \cite{HKr2},\cite{HKr3}.
The paper \cite{HKr2} is the main motivation of \cite{NP} which is the corner stone of our approach.

To summarize the results in this paper we start with the definition of Gowers norms.
Let $f:A\rightarrow\mathbb{C}$ be a bounded measurable function on a compact abelian group $A$. Let $\Delta_t f$ be the function with $\Delta_tf(x)=f(x)\overline{f(x+t)}$.
With this notation
$$\|f\|_{U_k}=\Bigl(\int_{x,t_1,t_2,\dots,t_k\in A}\Delta_{t_1}\Delta_{t_2}\dots\Delta_{t_k}f(x)~d\mu^{k+1}\Bigr)^{2^{-k}}$$
where $\mu$ is the normalized Haar measure on $A$.
These norms satisfy the inequality $\|f\|_{U_k}\leq\|f\|_{U_{k+1}}$.
It is easy to verify that
$$\|f\|_{U_2}=\Bigl(\sum_{\chi\in\hat{A}}|\lambda_\chi|^4\Bigr)^{1/4}$$
where $\lambda_\chi=(f,\chi)$ is the Fourier coefficient corresponding to the linear character $\chi$.
This formula explains the behavior of the $U_2$ norm in terms of ordinary Fourier analysis.
However if $k\geq 3$, ordinary Fourier analysis does not seem to give a good understanding of the $U_k$ norm.

Small $U_2$ norm of a function $f$ with $|f|\leq 1$ is equivalent with the fact that $f$ is ``noise'' or ``quasi random'' from the ordinary Fourier analytic point of view. This means that all the Fourier coefficients have small absolute value. Such a noise however can have a higher order structure measured by one of the higher Gowers norms.
Isolating the structured part from the noise is a central topic in higher order Fourier analysis.   
In $k$-th order Fourier analysis a function $f$ is considered to be quasi random if $\|f\|_{U_{k+1}}$ is small.
As we increase $k$, this notion of noise becomes stronger and stronger and so more and more functions are considered to be structured.

The prototype of a decomposition theorem into structured and quasi random parts is Szemer\'edi's famous regularity lemma for graphs \cite{Szem2}.
The regularity lemma together with an appropriate counting lemma is a fundamental tool in combinatorics.
It is natural to expect that a similar regularization corresponding to the $U_{k+1}$ norm is helpful in additive combinatorics.
One can state the graph regularity lemma as a decomposition theorem for functions of the form $f:V\times V\rightarrow\mathbb{C}$ with $\|f\|\leq 1$.
Roughly speaking it says that $f=f_s+f_e+r_r$ where $f_r$ has small cut norm, $f_e$ has small $L^1$ norm and $f_s$ is of bounded complexity. (All the previous norms are normalized to give $1$ for the constant $1$ function.) We say that $f_r$ has complexity $m$ if there is a partition of $V$ into $m$ almost equal parts such that $f_s(x,y)$ depends only on the partition sets containing $x$ and $y$. This can also be formulated in a more algebraic way. A complexity $m$ function on $V\times V$ is the composition of $\phi:V\times V\rightarrow [m]\times [m]$ ({\bf algebraic part}) with another function $f:[m]\times [m]\rightarrow\mathbb{C}$ ({\bf analytic part}) where $[m]$ is the set of first $m$ natural numbers and $\phi$ preserves the product structure in the sense that $\psi=g\times g$ for some map $g:V\rightarrow [m]$. 
In this language the requirement that the partition sets are of almost equal size translates to the condition that $\phi$ is close to be preserving the uniform measure. 

Based on this (with some optimism) one can expect that there is a regularity lemma corresponding to the $U_{k+1}$ norm of a similar form.
This means that a bounded (measurable) function $f$ on a finite (or more generally on a compact) abelian group is decomposable as 
$f=f_s+f_e+f_r$ where $\|f_r\|_{U_{k+1}}$ is small, $\|f_e\|_1$ is small and $f_s$ can be obtained as the composition of $\phi:A\rightarrow N$ (algebraic part) and $g:N\rightarrow\mathbb{C}$ (analytic part) where $\phi$ is some kind of algebraic morphism preserving an appropriate structure on $A$.
The function $\phi$ would correspond to a regularity partition and $g$ would correspond to a function associating probabilities with the partition sets.
As we could assume the almost equality of the partition sets in Szemer\'edi's lemma we also expect that $\phi$ can be required to satisfy some almost measure preserving property. 

In this paper we show that this optimistic picture is almost exactly true with some interesting additional features.
Based on the graph regularity lemma we would expect that $N$ is also an abelian group. However quite interestingly (except for the case $k=1$) abelian groups are not enough for this purpose. To get the regularity lemma for the $U_{k+1}$ norm we will need to introduce $k$-step nilspaces that are generalizations of abelian groups. It turns out that $k$-step nilspaces are forming a category and the morphisms are suitable for the purpose of regularization. (The need for extra structures is less surprising if we compare the $U_{k+1}$ regularity lemma with the regularization of $k+1$ uniform hypergraphs. Except for the case $k=1$, which is the graph case, partitions of the vertex set are not enough to regularize hypergraphs.)

Another interesting phenomenon is that topology comes into the picture.
Topology doesn't seem to play an important role in stating the regularity lemma for graphs. (Note that a connection of Szemer\'edi's regularity lemma to topology was pointed out in \cite{LSz4}) However, quite surprisingly, in the abelian group case even if we just want to regularize functions on finite abelian groups the target space $N$ of $\phi$ needs to be a compact topological nilsapce. Furthermore we will require that the function $g$ is ``smooth'' enough with respect to the topology on $N$. A possible way of doing it is to require that $g$ is continuous with bounded Lipschitz constant in some fixed metric on $N$. However the Lipschitz condition is not crucial in our approach. It can be replaced by almost any reasonable complexity notion. For example we can use an arbitrary ordering of an arbitrary countable $L^\infty$-dense set of continuous functions on $N$ and then we can require that $g$ is on this list with a bounded index.
  
\bigskip

To state our regularity lemma we will need the definition of nilspaces.
Nilspaces are common generalizations of abelian groups and nilmanifolds.
An abstract cube of dimension $n$ is the set $\{0,1\}^n$.
A cube of dimension $n$ in an abelian group $A$ is a function $f:\{0,1\}^n\rightarrow A$ which extends to an affine homomorphism (a homomorphism plus a translation) $f':\mathbb{Z}^n\rightarrow A$. Similarly, a morphism $\psi:\{0,1\}^n\rightarrow\{0,1\}^m$ between abstract cubes is a map which extends to an affine morphism from $\mathbb{Z}^n\rightarrow\mathbb{Z}^m$.

Roughly speaking, a nilspace is a structure in which cubes of every dimension are defined and they behave very similarly as cubes in abelian groups.
\begin{definition}[Nilspace axioms] A nilspace is a set $N$ and a collection $C^n(N)\subseteq N^{\{0,1\}^n}$ of functions (or cubes) of the form $f:\{0,1\}^n\rightarrow N$ such that the following axioms hold.
\begin{enumerate}
\item {\bf(Composition)} If $\psi:\{0,1\}^n\rightarrow\{0,1\}^m$ is a cube morphism and $f:\{0,1\}^m\rightarrow N$ is in $C^m(N)$ then the composition $\psi\circ f$ is in $C^n(N)$.
\item {\bf(Ergodictiry)} $C^1(N)=N^{\{0,1\}}$.
\item {\bf(Gluing)} If a map $f:\{0,1\}^n\setminus\{1^n\}\rightarrow N$ is in $C^{n-1}(N)$ restricted to each $n-1$ dimensional face containing $0^n$ then $f$ extends to the full cube as a map in $C^n(N)$.
\end{enumerate} 
\end{definition}

\bigskip

If $N$ is a nilspace and in the third axiom the extension is unique for $n=k+1$ then we say that $N$ is a $k$-step nilspace.
If a space $N$ satisfies the first axiom (but the last two are not required) then we say that $N$ is a {\bf cubespace}. 
A function $f:N_1\rightarrow N_2$ between two cubespaces is called a {\bf morphism} if $\phi\circ f$ is in $C^n(N_2)$ for every $n$ and function $\phi\in C^n(N_1)$.
The set of morphisms between $N_1$ and $N_2$ is denoted by $\Hom(N_1,N_2)$. With this notation $C^n(N)=\Hom(\{0,1\}^n,N)$.
If $N$ is a nilspace then every morphism $f:\{0,1\}^n\rightarrow\{0,1\}^m$ induces a map $\hat{f}:C^m(N)\rightarrow C^n(N)$ by simply composing $f$ with maps in $C^m(N)$. We say that $N$ is a compact Hausdorff nilspace if all the spaces $C^n(N)$ are compact Hausdorff spaces and $\hat{f}$ is continuous for every $f$.

The nilspace axiom system is a variant of the Host-Kra axiom system for parallelepiped structures \cite{HKr2}. In \cite{HKr2} the two step case is analyzed and it is proved that the structures are tied to two nilpotent groups. A systematic analysis of $k$-step nilspaces was carried out by O. Camarena and the author in \cite{NP}. Compact nilspaces get special attention in \cite{NP}. It will be important that the notion of Haar measure can be generalized for compact nilspaces. One of the main results in \cite{NP} says that compact nilspaces are inverse limits of finite dimensional ones and the connected components of a finite dimensional compact nilspace are nilmanifolds with cubes defined through a given filtration on the nilpotent group. It is crucial that a $k$-step compact nilspace $N$ can be built up using $k$ compact abelian groups $A_1,A_2,\dots,A_k$ as structure groups in a $k$-fold iterated abelian bundle. The nilspace $N$ is finite dimensional if and only if all the structure groups are finite dimensional or equivalently: the dual groups $\hat{A_1},\hat{A_2},\dots,\hat{A_n}$ are all finitely generated. 
It follows from the results in \cite{NP} that there are countably many finite dimensional $k$-step nilspaces up to isomorphism. An arbitrary ordering on them will be called a {\bf complexity notion}.

For every finite dimensional nilspace $N$ and natural number $n$ we fix a metrization of the weak convergence of probability measures on $C^n(N)$.
Let $M$ and $N$ be (at most) $k$-step compact nilspaces such that $N$ is finite dimensional. Let $\phi:M\rightarrow N$ be a continuous morphism and let us denote by $\phi_n:C^n(M)\rightarrow C^n(N)$ the map induced by $\phi$ using composition. The map $\phi$ is called $b$-balanced if the probability distribution of $\phi_n(x)$ for a random $x\in C^n(M)$ is at most $b$-far from the uniform distribution on $C^n(N)$ whenever $n\leq 1/b$.
Being well balanced expresses a very strong surjectivity property of morphisms which is for example useful in counting.

\begin{definition}[Nilspace-polynomials] Let $A$ be a compact abelian group. A function $f:A\rightarrow\mathbb{C}$ with $|f|\leq 1$ is called a $k$-degree, complexity $m$ and $b$-balanced nilspace-polynomial if
\begin{enumerate}
\item $f=\phi\circ g$ where $\phi:A\rightarrow N$ is a continuous morphism of $A$ into a finite dimensional compact nilspace $N$,
\item $N$ is of complexity at most $m$,
\item $\phi$ is $b$-balanced,
\item $g$ is continuous with Lipschitz constant $m$. 
\end{enumerate}
\end{definition}

Note, that (as it will turn out) a nilspace-polynomial on a cyclic group is polynomial nilsequence with an extra periodicity property.
Now we are ready to state the decomposition theorem.

\begin{theorem}[Regularization]\label{reglem} Let $k$ be a fixed number and $F:\mathbb{R}^+\times\mathbb{N}\rightarrow\mathbb{R^+}$ be an arbitrary function. Then for every $\epsilon>0$ there is a number $n=n(\epsilon,F)$ such that for every measurable function $f:A\rightarrow\mathbb{C}$ on a compact abelian group $A$ with $|f|\leq 1$ there is a decomposition $f=f_s+f_e+f_r$ and number $m\leq n$ such that the following conditions hold.
\begin{enumerate}
\item $f_s$ is a degree $k$, complexity $m$ and $F(\epsilon,m)$-balanced nilspace-polynomial,
\item $\|f_e\|_1\leq\epsilon$,
\item $\|f_r\|_{U_{k+1}}\leq F(\epsilon,m)$~,~$|f_r|\leq 1$ and $|(f_r,f_s+f_e)|\leq F(\epsilon,m)$.
\item $|\|f_s+f_e\|_{U_{k+1}}-\|f\|_{U_{k+1}}|\leq F(\epsilon,m)$
\end{enumerate}
\end{theorem}

\begin{remark} The Gowers norms can also be defined for functions on $k$-step compact nilspaces. It makes sense to generalize our resuts from abelian groups to nilspaces. Almost all the proofs are essentially the same. This shows that the (algebraic part) of $k$-th order Fourier analysis deals with continuous functions between $k$-step nilspaces. 
\end{remark}

Note that various other conditions could be put on the list in theorem \ref{reglem}.
For example the proof shows that $f_s$ looks approximately like a projection of $f$ to a $\sigma$-algebra. This imposes strong restrictions on the value distribution of $f_s$ in terms of the value distribution of $f$. 

Theorem \ref{reglem} implies an inverse theorem for the $U_{k+1}$-norm.
It says that if $\|f\|_{U_{k+1}}$ is separated from $0$ then it correlates with a bounded complexity $k$-degree nilspace polynomial $\phi$.
(We can also require the function $\phi$ to be arbitrary well balanced in terms of its complexity but we omit this from the statement to keep it simple.)

\begin{theorem}[General inverse theorem for $U_{k+1}$]\label{invthem} Let us fix a natural number $k$. For every $\epsilon>0$ there is a number $n$ such that if $\|f\|_{U_{k+1}}\geq\epsilon$ for some measurable $f:A\rightarrow\mathbb{C}$ on a compact abelian group $A$ with $|f|\leq 1$ then $(f,g)\geq\epsilon^{2^k}/2$ for some degree $k$, complexity at most $n$ nilspace-polynomial.
\end{theorem}

Note that this inverse theorem is exact in the sense that if $f$ correlates with a bounded complexity nilspace polynomial then its Gowers norm is separated from $0$.

\bigskip

A strengthening of the decomposition theorem \ref{reglem} and inverse theorem \ref{invthem} deals with the situation when the abelian groups are from special families. 
For example we can restrict our attention to elementary abelian $p$-groups with a fixed prime $p$. Another interesting case is the set of cyclic groups or bounded rank abelian groups.
It also make sense to develop a theory for one particular infinite compact group like the circle group $\mathbb{R}/\mathbb{Z}$.
It turns out that in such restricted families of groups we get restrictions on the structure groups of the nilspaces we have to use in our decomposition theorem.
To formulate these restrictions we need the next definition.

\begin{definition} Let $\mathfrak{A}$ be a family of compact abelian groups. We denote by $(\mathfrak{A})_k$ the set of finitely generated groups that arise as subgroups of $\hat{\bA}_k$ where $\bA$ is some ultra product of groups in $\mathfrak{A}$ and $\hat{\bA}_k$ is the $k$-th order dual group of $\bA$ in the sense of \cite{Sz2}. 
\end{definition}

A good description of $(\mathfrak{A})_k$ is available in \cite{Sz2} but in this paper we focus on the following cases.

\begin{enumerate}
\item {\bf (Bounded exponent and characteristic $p$)}~If $\mathfrak{A}$ is the set of finite groups of exponent $n$ then $(\mathfrak{A})_k=\mathfrak{A}$ for every $k$. In particular if $n=p$ prime then $\mathfrak{A}$ and $(\mathfrak{A})_k$ are just the collection of finite dimensional vector spaces over the field with $p$ elements.
\item {\bf (Bounded rank)}~If $\mathfrak{A}$ is the set of finite abelian groups of rank at most $d$ then $(\mathfrak{A})_1$ is the collection of all finitely generated abelian groups but $(\mathfrak{A}_k)$ contains only free abelian groups if $k\geq 2$. The case $d=1$ is the case of cyclic groups.
\item {\bf (Characteristic $0$)}~If $\mathfrak{A}$ is a family of finite abelian groups in which for every natural number $n$ there are only finitely many groups with order divisible by $n$ then $(\mathfrak{A}_k)$ contains only free abelian groups for every $k$.
\item {\bf (Circle group)}~If $\mathfrak{A}$ contains only the circle group $\mathbb{R}/\mathbb{Z}$ then $(\mathfrak{A})_k$ contains only free abelian groups for every $k$.
\end{enumerate} 

\begin{definition} Let $\mathfrak{A}$ be a family of compact abelian groups. A $k$-step $\mathfrak{A}$-nilspace is a finite dimensional nilspace with structure groups $A_1,A_2,\dots,A_k$ such that $\hat{A_i}\in(\mathfrak{A})_i$ for every $1\leq i\leq k$.
A nilspace polynomial is called $\mathfrak{A}$-nilspace polynomial if the corresponding morphism goes into an $\mathfrak{A}$-nilspace.\end{definition}
Then we have the following.

\begin{theorem}[Regularization in special families]\label{restreg} Let $\mathfrak{A}$ be a set of compact abelian groups. Then Theorem \ref{reglem} restricted to functions on groups from $\mathfrak{A}$ is true with the stronger implication that the structured part $f_s$ is an $\mathfrak{A}$-nilspace polynomial.
\end{theorem} 

\begin{theorem}[Specialized inverse theorem for $U_{k+1}$]\label{restinv} Let $\mathfrak{A}$ be a set of compact abelian groups. Then for functions on groups in $\mathfrak{A}$ theorem \ref{invthem} holds with $\mathfrak{A}$-nilspace polynomials.
\end{theorem}

This theorem shows in particular that if $\mathfrak{A}$ is the set of elementary abelian $p$-groups then the nilspace used in the regularization is a finite nilspace such that all the structure groups are elementary abelian.
More generally if $\mathfrak{A}$ is the set of abelian groups in which the order of every element divides a fixed number $n$ (called groups of exponent $n$) then $\mathfrak{A}$-nilspaces (used in the regularization) are finite and all the structure groups have exponent $n$. 

In the $0$ characteristic case $\mathfrak{A}$-nilspaces are $k$-step nilmanifold with a given filtration. This will help us to give a generalization of the Green-Tao-Ziegler theorem \cite{GTZ} for a multidimensional setting.
In the case of the circle group, again we only get $k$-step nilmanifolds.

We will see that our methods give an even stronger from of theorem \ref{restreg} which seems to be helpful in the bounded exponent case but to formulate this we need to use ultra products more deeply.
Before doing so we first highlight our results about counting and limit objects for function sequences.

\bigskip

Roughly speaking, counting deals with the density of given configurations in functions on compact abelian groups. We have two goals with counting. One is to show that our regularity lemma is well behaved with respect to counting and the second goal is
to show that function sequences in which the density of every fixed configuration converges have a nice limit object which is a measurable function on a nilspace. This fits well into the recently developed graph and hypergraph limit theories \cite{LSz1},\cite{BCLSV1},\cite{LSz3},\cite{LSz4},\cite{ESz}.
 
Counting in compact abelian groups has two different looking but equivalent interpretations. One is about evaluating certain integrals and the other is about the distribution of random samples from a function.
Let $f:A\rightarrow\mathbb{C}$ be a bounded function and the compact group $A$.
An integral of the form 
$$\int_{x,y,z\in A} f(x+y)f(x+z)f(y+z)~d\mu^3$$
can be interpreted as the triangle density in the weighted graph $M_{x,y}=f(x+y)$.
Based on this connection, evaluating such integrals can be called counting in $f$.
Note that one might be interested in more complicated integrals like this:
$$\int_{x,y,z} f(x)\overline{f(z)}^5f(x+y+z)\overline{f(x+y)}f(y+z)^2~d\mu^3$$
where conjugations and various powers appear. 
It is clear that as long as the arguments are sums of different independent variables then all the above integrals can be obtained from knowing the seven dimensional distribution of
\begin{equation}\label{jointdis}
(f(x),f(y),f(z),f(x+y),f(x+z),f(y+z),f(x+y+z))\in\mathbb{C}^7
\end{equation}
where $x,y,z$ are randomly chosen elements form $A$ with respect to the Haar measure.
One can think of the above integrals as multi dimensional moments of the distribution in (\ref{jointdis}).
We will say that such a moment (or the integral itself) is simple if it does not contain higher powers. (We allow conjugation in simple moments.)
We will see that there is a slight, technical difference between dealing with simple moments and dealing with general moments.

Every moment can be represented as a colored (or weighted) hypergraph on the vertex set $\{1,2,\dots,n\}$ where $n$ is the number of variables and an edge $S\subseteq\{1,2,\dots,n\}$ represents the term $f(\sum_{i\in S}x_i)$ in the product.
The color of an edge tells the appropriate power and conjugation for the corresponding term.
The degree of a moment is the maximal size of an edge minus one in this hypergraph.
Let $\mathcal{M}$ denote the set of all simple moments and let $\mathcal{M}_k$ denote the collection of simple moments of degree at most $k$. 
We will denote by $D_n(f)$ the joint distribution of $\{f(\sum_{i\in S} x_i)\}_{S\subset[n]}$ where $[n]=\{1,2,\dots,n\}$.
It is a crucial fact that all the moments and the distributions $D_n$ can also be evaluated for functions on compact nilspaces with a distinguished element $0$. 
We call nilspaces with such an element ``rooted nilspaces''.
Let $N$ be a rooted nilspace. If we choose a random $n$-dimensional cube $c:\{0,1\}^n\rightarrow N$ in $C^n(N)$ with $f(0^n)=0$ then the joint distribution of the values $\{f(c(v))\}_{v\in\{0,1\}^n}$ gives the distribution $D_n(f)$. 
Using this we will prove the next theorem.

\begin{theorem}[Limit object I.]\label{simplim} Assume that $\{f_i\}_{i=1}^\infty$ is a sequence of uniformly bounded measurable functions on the compact abelian groups $\{A_i\}_{i=1}^\infty$. Then if $\lim_{i\to\infty} M(f_i)$ exists for every $M\in\mathcal{M}$ then there is a measurable function (limit object) $g:N\rightarrow\mathbb{C}$ on a compact rooted nilspace $N$ such that $M(g)=\lim_{i\to\infty}M(f_i)$ for $M\in\mathcal{M}$.
\end{theorem}

\begin{corollary}[Limit object II.] Let $k$ be a fixed natural number. Assume that $\{f_i\}_{i=1}^\infty$ is a sequence of uniformly bounded measurable functions on the compact abelian groups $\{A_i\}_{i=1}^\infty$. Then if $\lim_{i\to\infty} M(f_i)$ exists for every $M\in\mathcal{M}_k$ then there is a measurable function (limit object) $g:N\rightarrow\mathbb{C}$ on a compact $k$-step rooted nilspace $N$ such that $M(g)=\lim_{i\to\infty}M(f_i)$ for $M\in\mathcal{M}_k$.
\end{corollary}

Let $\mathcal{P}_r$ denote the space of Borel probability distributions supported on the set $\{x:|x|\leq r\}$ in $\mathbb{C}$.

\begin{theorem}[Limit object III.]\label{genlim} Assume that $\{f_i\}_{i=1}^\infty$ is a sequence of functions with $|f_i|\leq r$ on the compact abelian groups $\{A_i\}_{i=1}^\infty$. Then if $\lim_{i\to\infty} D_k(f_i)$ exists for every $k\in\mathbb{N}$ then there is a measurable function (limit object) $g:N\rightarrow\mathcal{P}_r$ on a compact rooted nilspace $N$ such that $D_k(g)=\lim_{i\to\infty}D_k(f_i)$ for $k\in\mathbb{N}$.
\end{theorem}
Let us observe that theorem \ref{genlim} implies the other two.

We devote the last part of the introduction to our main method and the simple to state theorem \ref{ultreg} on ultra product groups which implies almost everything in this paper.
Let $\{A_i\}_{i=1}^\infty$ be a sequence of compact abelian groups and let $\bA$ be their ultra product.
Our strategy is to develop a theory for the Gowers norms on $\bA$ and then by indirect arguments we translate it back to compact groups.
First of all note that $U_{k+1}$ is only a semi norm on $\bA$. It was proved in \cite{Sz1} that there is a (unique) maximal $\sigma$-algebra $\mathcal{F}_k$ on $\bA$ such that $U_{k+1}$ is a norm on $L^\infty(\mathcal{F}_k)$. 
It follows that every function $f\in L^\infty(\bA)$ has a unique decomposition as $f_s+f_r$ where $\|f_r\|_{U_{k+1}}=0$ and $f_s$ is measurable in $\mathcal{F}_k$. 
This shows that on the ultra product $\bA$ it is simple to separate the structured part of $f$ from the random part. 

The question remains how to describe the structured part in a meaningful way.
It turns out that to understand this we need to go beyond measure theory and use topology. 
Note that the reason for this is not that the groups $A_i$ are already topological. Even if $\{A_i\}_{i=1}^\infty$ is a sequence of finite groups, topology will come into the picture in the same way.  

We will make use of the fact that $\bA$ has a natural $\sigma$-topology on it. A $\sigma$-topology is a weakening of ordinary topology where only countable unions of open sets are required to be open.
The structure of $\mathcal{F}_1$, which is tied to ordinary Fourier analysis, sheds light on how topology comes into the picture. It turns out that $\mathcal{F}_1$ can be characterized as the smallest $\sigma$-algebra in which all the continuous surjective homomorphisms $\phi:\bA\rightarrow G$ are measurable where $G$ is a compact Hausdorff abelian group.
In other words the ordinary topological space $G$ appears as a factor of the $\sigma$-topology on $\bA$.
The next theorem explains how nilspaces enter the whole topic:

\begin{theorem}[Characterization of $\mathcal{F}_{k+1}$.] The $\sigma$-algebra $\mathcal{F}_k$ is the smallest $\sigma$-algebra in which all continuous morphisms $\phi:\bA\rightarrow N$ are measurable where $N$ is a compact $k$-step nilspace.
\end{theorem}

Another, stronger formulation of the previous theorem says that every separable $\sigma$-algebra in $\mathcal{F}_k$ is measurable in a  $k$-step compact, Hausdorff nilspace factor of $\bA$.
We will see later that we can also require a very strong measure preserving property for the nilspace factors $\phi:\bA\rightarrow N$. This will also be crucial in the proofs.
As a corollary we have a very simple regularity lemma on the ultra product group $\bA$.

\begin{theorem}[Ultra product regularity lemma]\label{ultreg} Let us fix a natural number $k$. Let $f\in L^\infty(\bA)$ be a function. Then there is a unique (orthogonal) decomposition $f=f_s+f_r$ such that $\|f_r\|_{U_{k+1}}=0$ and $f_s$ is measurable in a $k$-step compact nilspace factor of $\bA$.
\end{theorem}

\subsection{A multidimensional generalization of the Green-Tao-Ziegler theorem.}\label{cyclic}

The following notion of a polynomial map between groups was introduced by Leibman \cite{Lei}.

\begin{definition}[Polynomial map between groups] A map $\phi$ of a group $G$ to a group $F$ is said to be polynomial of degree $k$ if it trivializes after $k+1$ consecutive applications of the operator $D_h,~h\in G$ defined by $D_h\phi(g)=\phi(g)^{-1}\phi(gh).$
\end{definition}

Our goal in this chapter is to relate nilspace polynomials on cyclic groups to Leibman type polynomials.
As a consequence we will obtain a new proof of the inverse theorem by Green, Tao and Ziegler for cyclic groups.
We will need a few definitions and lemmas. 

\begin{lemma}\label{polyab} Let $A$ be an abelian group. Then the set of at most degree $k$ polynomials form $\mathbb{Z}^n$ to $A$ is generated by the functions of the form
\begin{equation}\label{abpoly}
f(x_1,x_2,\dots,x_n)=a\prod_{i=1}^n{{x_i}\choose{n_i}}
\end{equation}
where $a\in A$ and $\sum_{i=1}^n n_i\leq k$.
(We use additive notation here)
\end{lemma}

\begin{proof} We go by induction on $k$. The case $k=0$ is trivial. Assume that it is true for $k-1$. Let $g_1,g_2,\dots,g_n$ be the generators of $\mathbb{Z}^n$. If $\phi:\mathbb{Z}^n\rightarrow A$ is a polynomial map then $$\omega(y_1,y_2,\dots,y_k)=D_{y_1}D_{y_2}\dots D_{y_k}\phi$$ is a symmetric $k$-linear form on $\mathbb{Z}^n$. We claim that there is map $\phi'$ which is generated by the functions in (\ref{abpoly}) and whose $k$-linear form is equal to $\omega$. Let $f:\otimes^k(\mathbb{Z}^n)\rightarrow A$ be a homomorphism representing $\omega$. Then $f$ is generated by homomorphisms $h$ such that $h(g_{j_1}\otimes g_{j_2}\otimes\dots g_{j_k})=a$ for some indices $j_1,j_2,\dots,j_k$ (and any ordering of them) and take $0$ on any other tensor products of generators. It is enough to represent such an $h$ by a function of the form (\ref{abpoly}). 
It is easy to see that if $n_i$ is the multiplicity of $i$ among the indices $\{j_r\}_{r=1}^k$ then (\ref{abpoly}) gives a polynomial whose multi linear form is represented by $h$. 

The difference $\phi-\phi'$ has a trivial $k$-linear form which shows that it is a $k-1$ dimensional polynomial and then we use induction the generate $\phi-\phi'$.
\end{proof}

The next definition is by Leibman \cite{Lei2}.

\begin{definition} Let $F$ be a $k$-nilptent group with filtration $\mathcal{V}=\{F_i\}_{i=0}^k$ with $F=F_0$, $F_{i+1}\subseteq F_i$, $F_k=\{1\}$ and $[F_i,F]\subseteq F_{i+1}$ if $i<k$. A map $\phi:G\rightarrow F$ is a $\mathcal{V}$-polynomial if $\phi$ modulo $F_i$ is a polynomial of degree $i$.
\end{definition}

It is proved in \cite{Lei2} that $\mathcal{V}$ polynomials are closed under multiplication. 
Related to the filtration $\mathcal{V}$ (using the notation of the previous definition) we can also define a nilspace structure on $F$. A map $f:\{0,1\}^n\rightarrow F$ is in $C^n(F,\mathcal{V})$ if it can be obtained from the constant $1$ map in a finite process where in each step we choose a natural number $1\leq i\leq k$ and an element $x\in N_i$ and then we multiply the value of $f$ on an $i+1$ codimensional face of $\{0,1\}^n$ by $x$. 

If $H\subset F$ is a subgroup then there is an inherited nilspace structure on the left coset space $M$ of $H$ in $F$.
This is obtained by composing all the maps in $C^n(F)$ with $F\rightarrow M$.
Note that the structure groups of $M$ are the groups $A_i=F_iH/F_{i+1}H$. The fact that $\mathcal{V}$ is a filtration implies that $[F_iH,F_iH]\subseteq F_{i+1}H$ and so $A_i$ is an abelian group.

\begin{lemma}\label{felemelo} Let $\phi:\mathbb{Z}^n\rightarrow M$ be a morphism. Then there is a lift $\psi:\mathbb{Z}^n\rightarrow F$ such that $\psi$ is a $\mathcal{V}$-polynomial and $\psi$ composed with the projection $F\rightarrow M$ is equal to $\phi$.
\end{lemma}

\begin{proof} Using induction of $j$ we show the statement for maps whose image is in $F_{k-j}H$.
If $j=0$ then $\phi$ is a constant map and then the statement is trivial. Assume that we have the statement for $j-1$ and assume that the image of $\phi$ is in $F_{k-j}H$.
The cube preserving property of $\phi$ shows that $\phi$ composed with the factor map $F_{k-j}H\rightarrow F_{k-j}H/F_{k-j+1}H=A_{k-j}$ is a degree $k-j+1$ polynomial map $\phi_2$ of $\mathbb{Z}^n$ into the abelian group $A_{k-j}$. 

We have from lemma \ref{polyab} that using multiplicative notation
\begin{equation}\label{polfor}
\phi_2(x_1,x_2,\dots,x_n)=\prod_{t=1}^m a_t^{f_t(x_1,x_2,\dots,x_n)}
\end{equation}
where $a_t\in A_{k-j}$ and $f_t$ is an integer valued polynomial of degree at most $k-j+1$ for every $t$. Let us choose elements $b_1,b_2,\dots,b_m$ in $F_{k-j}$ such that their images in $A_{k-j}$ are $a_1,a_2,\dots,a_m$.
Let us define the function $\alpha:\mathbb{Z}\rightarrow F$ given by the formula (\ref{polfor}) when $a_t$ is replaced by $b_t$. The map $\alpha$ is a $\mathcal{V}$-polynomial.

Since $\phi$ maps to the left cosets of 
$H$ it makes sens to multiply $\phi$ by $\alpha^{-1}$ from the left. It is easy to see that the new map $\gamma=\alpha^{-1}\phi$ is a morphism of $\mathbb{Z}$ to $F_{k-j+1}H$ and thus by induction it can be lifted to a $\mathcal{V}$ polynomial $\delta$. Then we have that $\alpha\delta$ is a lift of $\phi$ to a polynomial map. 
\end{proof}

\begin{corollary} If $A$ is a finite abelian group and $f:A\rightarrow M$ is a morphism then for every homomorphism $\beta:\mathbb{Z}^n\rightarrow A$ there is a degree $k$ polynomial map $\phi:\mathbb{Z}^n\rightarrow F$ such that $\phi$ composed with the factor map $F\rightarrow M$ is the same as $\beta$ composed with $f$.
\end{corollary}

\begin{definition}[$d$-dimensional polynomial nilsequence] Assume that $F$ is a connected $k$-nilpotent Lie group with filtration $\mathcal{V}$ and $\Gamma$ is a co-compact subgroup of $F$. Assume that $M$ is the left coset space of $\Gamma$ in $F$. Then a map $h:\mathbb{Z}^d\rightarrow\mathbb{C}$ is called a $d$-dimensional polynomial nilsequence (corresponding to $M$) if there is a polynomial map $\phi:\mathbb{Z}^d\rightarrow F$ of degree $k$ and a continuous Lipschitz function $g:M\rightarrow\mathbb{C}$ such that $h$ is the composition of $\phi$, the projection $F\rightarrow M$ and $g$. 
The complexity of such a nilsequence is measured by the maximum of $c$ and the complexity of $M$.
\end{definition}

Let $\mathfrak{A}$ be a $0$-characteristic family of abelian groups. Then all the structure groups of $\mathfrak{A}$-nilspaces are free abelian groups. By \cite{NP} we get that $\mathfrak{A}$-nilspaces are nilmanifolds with a given filtration. 
From this and theorem \ref{restinv} we obtain the following consequence.
\begin{theorem}[polynomial nilsequence inverse theorem]\label{PNIT} Let $\mathfrak{A}$ be a $0$ characteristic family of finite abelian groups. Let us fix a natural number $k$. For every $\epsilon>0$ there is a number $n$ such that if $\|f\|_{U_{k+1}}\geq\epsilon$ for some measurable $f:A\rightarrow\mathbb{C}$ with $|f|\leq 1$ on $A\in\mathfrak{A}$ with $d$-generators $a_1,a_2,\dots,a_d$ then $(f,g)\geq\epsilon^{2^k}/2$ such that $g(n_1a_1+n_2a_2+\dots+n_da_d)=h(n_1,n_2,\dots,n_d)$ for some $d$-dimensional polynomial nilsequence $h$ of complexity at most $n$.
\end{theorem}

Note that the above theorem implies an interesting periodicity since the defining equation of $g$ is true for every $d$-tuple $n_1,n_2,\dots,n_d$ of integers. 
Using theorem \ref{PNIT} we obtain the Green-Tao-Ziegler inverse theorems for functions $f:[N]\rightarrow\mathbb{C}$ with $|f|\leq 1$. Their point of view is that if we put the interval $[N]$ into a large enough cyclic group (say of size $m>N2^{k+1}$) then the normalized version of $\|f\|_{U_{k+1}}$ does not depend on the choice of $m$. The proper normalization is to divide with the $U_{k+1}$-norm of the characteristic function on $1_{[N]}$. 

To use theorem \ref{PNIT} in this situation we need to make sure that $m$ is not too big and that it has only large prime divisors. This can be done by choosing a prime between $N2^{k+1}$ and $N2^{k+2}$.
Then we can apply theorem \ref{PNIT} for the family of cyclic groups of prime order which is clearly a $0$ characteristic family.
What we directly get is that $f$ correlates with a bounded complexity polynomial nil-sequence of degree $k$.
This seems to be weaker then the Green-Tao-Ziegler theorem because they obtain the correlation with a linear nil-sequence. However in the appendix of \cite{GTZ} it is pointed out that the two versions are equivalent.

Form theorem \ref{PNIT} we can also obtain a $d$-dimensional inverse theorem for functions of the form $f:[N]^d\rightarrow\mathbb{C}$ with $|f|\leq 1$. Here we use the family of $d$-th direct powers of cyclic groups with prime order. 

\begin{theorem}[Multi dimensional inverse theorem] Let us fix two natural numbers $d,k>0$. Then for every $\epsilon>0$ there is a number $n$ such that for every function $f:[N]^d\rightarrow\mathbb{C}$ with $\|f\|_{U_{k+1}}\geq\epsilon$ there is a $d$-dimensional polynomial nil-sequence $h:\mathbb{Z}^d\rightarrow\mathbb{C}$ of complexity at most $n$ and degree $k$ such that $(f,h)\geq \epsilon^{2^k}/2$. 
\end{theorem}

Note that $(f,h)$ is the scalar product normalized as as $(f,h)=N^{-d}\sum_{v\in [N]^d}(f(v)\overline{h(v)})$.

\subsection{A curious example using the Heisenberg group}\label{heis}

In this chapter we discuss an example which highlights a difference between the nilseqence approach used in \cite{GTZ} and the nilspace-polynomial approach used in the present paper. 

Let $e(x)=e^{x2\pi i}$. For an integer $1<t<m$ we introduce the function $f:\mathbb{Z}_m\rightarrow\mathbb{C}$ defined by $f(k)=\lambda^{k^2}$ where $\lambda=e(t/m^2)$ and $k=0,1,2,\dots,m-1$. Note that this function does not ``wrap around'' nicely like a more simple quadratic function of the form $k\mapsto\epsilon^{k^2}$ where $\epsilon$ is an $m$-th root of unity. This means that to define $f$ we need to choose explicit integers to represent the residue classes modulo $m$.
On the other hand it can be seen that $\|f\|_{U_3}$ is uniformly separated from $0$ so it has some quadratic structure.

In the nilsequence approach this function is not essentially different from the case of $k\mapsto\epsilon^{k^2}$.
However in our approach we are more sensitive about the periodicity issue since we want to establish $f$ through a very rigid algebraic morphism which uses the full group structure of $\mathbb{Z}_m$.
We will show that the quadratic structure of $f$ is tied to a nilspace morphism $\phi$ which maps $\mathbb{Z}_m$ into the Heisenberg nilmanifold.

The Heisenberg group $H$ is the group of three by three upper uni-triangular matrices with real entries.
Let $\Gamma\subset H$ be the set of integer matrices in $H$. It can be seen that $\Gamma$ is a co-compact subgroup. The left coset space $N=\{g\Gamma|g\in H\}$ of $\Gamma$ in $H$ is the Heisenberg nilmanifold.
Let $M\in H$ be the following matrix:
\[ \left( \begin{array}{ccc}
1 & 2t/m & t/m^2\\
0 & 1 & 1/m \\
0 & 0 & 1 \end{array} \right)\]
then

\[M^k= \left( \begin{array}{ccc}
1 & 2kt/m & k^2t/m^2\\
0 & 1 & k/m \\
0 & 0 & 1 \end{array} \right)\]
In particular $M^m$ is an integer matrix.
This implies that the map $\tau:k\rightarrow M^k\Gamma$ defines a periodic morphism from $\mathbb{Z}$ to $N$.
Since the period length is $m$, it defines a morphism $\phi:\mathbb{Z}_m\rightarrow N$.

Let $D$ be the set of elements in $H$ in which all entries are between $0$ and $1$. The set $D$ is a fundamental domain for $\Gamma$. We can define a function on $N$ by representing it on the fundamental domain.
Let $g:D\rightarrow\mathbb{C}$ be the function $A\rightarrow e(A_{1,3})$ where $A_{1,3}$ is the upper-right corner of the matrix $A$. 

We compute $g(\tau(k))=g(M^k\Gamma)$ by multiplying $M^k$ back into the fundamental domain $D$.
Since $g(M^k\Gamma)$ is periodic we can assume that $0\leq k<m$.
Let us multiply $M^k$ from the right by

\[ \left( \begin{array}{ccc}
1 & -\lfloor 2kt/m\rfloor & -\lfloor k^2t/m^2\rfloor\\
0 & 1 & 0\\
0 & 0 & 1 \end{array} \right)\]

We get 

\[ \left( \begin{array}{ccc}
1 & \{ 2kt/m\} & \{k^2t/m^2\}\\
0 & 1 & k/m\\
0 & 0 & 1 \end{array} \right)\in D\]

So the value of $g$ on $\tau(k)$ is $e(\{k^2t/m^2\})=e(k^2t/m^2)$.

\subsection{Higher order Fourier analysis on the cirlce}

In this chapter we sketch a consequence of our results when specialized to the circle grouop $C=\mathbb{R}/\mathbb{Z}$.
In a separate paper we will devote more attention to this important case.
Since the circle falls in to the $0$-characteristic case, higher order Fourier analysis on the circle deals with continuous morphisms from $C$ to nilmanifolds with a prescribed filtration.

Let $F$ be a $k$-nilpotent Lie group and $\Gamma$ be a lattice in $F$. Let $N$ be the left coset space of $\Gamma$ in $F$. We define the nilspace structure on $N$ (as in chapter \ref{cyclic}) using a filtration $\mathcal{V}$ on $F$.
Assume that $\phi:C\rightarrow N$ is a continuous nilspace morphism.
Our goal is to show that $\phi$ can be lifted to a continuous polynomial map $f:\mathbb{R}\rightarrow F$.

\begin{lemma}\label{lin} Let $g:C\rightarrow C$ be a continuous polynomial map. Then $g$ is linear. 
\end{lemma}

\begin{proof} In this proof we will think of $C$ as the complex unit circle. Assume by contradiction that $g$ is not linear. Then by repeatedly applying operators $D_h$ to $g$ we can get a non-linear quadratic function.
This means that it is enough to get a contradiction if $g$ is quadratic.
In this case $D_h(g)$ is a linar map that depends continuously on $h\in C$ in the $L_2$ norm. On the other hand  $D_h(g)$ is a linear character times a compex number from the unit circle. The characters are orthogonal to each other and so the character corresponding to $D_h(g)$ has to be the same for every $h$. This is only possible if it is the trivial character but then $g$ is linear.
\end{proof}

As a consequence we obtain that a polynomial map $g:C\rightarrow C^n$ is also linear.
Using lemma \ref{lin} and (essentially) the same argument as in lemma \ref{felemelo} we obtain the following.

\begin{lemma} If $\phi:C\rightarrow N$ is a morphism then it has a lift $f:\mathbb{R}\rightarrow F$ such that $f=f_1f_2\dots f_k$ where $f_i$ is a continuous homomorphism for every $i$.
\end{lemma}

The main difference in the proof is that we have to use some easy Lie theory to show that one parameter subgroups in a factor of a nilpotent Lie-group can be lifted to one parameter subgroups.

Using this we also get an ``interval'' version of the Green-Tao-Ziegler theorem for the $U_{k+1}([0,1])$ norm.
Functions on the interval $[0,1]$ can be represented in a large enough Cyclic group say $\mathbb{R}/2^{k+1}\mathbb{Z}$.
We obtain that if $f:[0,1]\rightarrow\mathbb{C}$ is a measurable function with $|f|\leq 1$ and $\|f\|_{U_{k+1}}$ is separated from $0$ then $f$ correlates with a continuous bounded compexity polynomial nilsequence.

\subsection{Limits of functions on groups, (hyper)-graph limits and examples}

The limit notion for functions on abelian groups is closely tied to the graph and hypergraph limit languages \cite{LSz1},\cite{BCLSV1},\cite{LSz3},\cite{LSz4},\cite{ESz}.
Assume that $S$ is a subset in a finite abelian group $A$. We create a symmetric $k$-uniform hypergraph $H_k(S)$
on the vertex set $A$ with edge set
$$E=\{(x_1,x_2,\dots,x_k)|\sum_{i=1}^k x_i\in S\}.$$
We can look at these graphs as hypergraph versions of Cayley graphs.
Let $\{S_i\subseteq A_i\}_{i=1}^\infty$ be a sequence of subsets in abelian groups and let $1_{S_i}$ be the characteristic function of $S_i$. 
It is easy to see that if the function sequence $\{1_{S_i}\}_{i=1}^\infty$ is convergent then the hypergraph sequence $\{H_k(S_i)\}_{i=1}^\infty$ is also convergent in the sense of \cite{ESz} for every fixed $k$.
Without going into the details we mention that the limit object of $\{H_k(S_i)\}_{i=1}^\infty$ can be easily constructed from the limit object for $\{1_{S_i}\}_{i=1}^\infty$. 

The case $k=2$ is already interesting. The graph $H_2(S)$ is an undirected Cayley graph in which two points $x$ and $y$ in $A$ are connected if $x+y\in S$. Note that according to the original definition of a Cayley graph, $x$ and $y$ are connected by a directed edge if $x-y\in S$. Let us denote this version by $C(S)$. 
Assume that $f_{S_i}$ is convergent. It follows from our method that the limit object is a measurable function $f:A\rightarrow [0,1]$ on a compact abelian group $A$. It is clear that the two variable function $W(x,y)=f(x+y)$ represents the limit of the sequence $\{H_2(S_i)\}_{i=1}^\infty$ in the sense of \cite{LSz1}. Similarly $W'(x,y)=f(x-y)$ represent the limit of $\{C(S_i)\}_{i=1}^\infty$.
Note that the limit of Cayley graphs in the non-commutative case is determined in \cite{Sz4}.

We show two interesting examples.

\noindent{\bf Example 1.}~Recall that $e(x)=e^{x2\pi i}$. Let $f_n:\mathbb{Z}_n\rightarrow\mathbb{C}$ be the function defined by $f(k)=e(k/n)+e(a_nk/n)$ such that $a_n$ is a ``generic enough'' residue class modulo $n$. This means that if $m_1,m_2$ are any integers not both zero and smaller than $t_n$ in the absolute value then $m_1a_n+m_2$ is not $0$ modulo $n$. We assume that $a_n$ is chosen so that $\lim_{n\to\infty}t_n=\infty$.
It can be proved that $\{f_n\}_{n=1}^\infty$ is a convergent sequence and the limit is the function on the two dimensional torus $f:(\mathbb{R}/\mathbb{Z})^2\rightarrow\mathbb{C}$ defined by $f(x,y)=e(x)+e(y)$.

\medskip

\noindent{\bf Example 2.}~Let $0<\alpha<1$ be an irrational number. We define $f_n:\mathbb{Z}_n\rightarrow\mathbb{C}$ as the function $k\rightarrow \lambda^{k^2}$ where $\lambda=e(\lfloor\alpha n\rfloor/n^2)$ and $k=0,1,2,\dots,n-1$.
The limit object is the measurable function $g$ on the Heisenberg nilmanifold defined in chapter \ref{heis}

\section{The theory of nilspace factors}

\subsection{Nilspaces, abelian bundles and three-cubes}

Every abelian group $A$ has a natural nilspace structure in which cubes are maps $f:\{0,1\}^n\rightarrow A$ which extend to affine homomorphisms $f':\mathbb{Z}^n\rightarrow A$. We refer to this as the linear structure on $A$. For every natural number $k$ we can define a nilspace structure $\mathcal{D}_k(A)$ on $A$ that we call the $k$-degree structure.
A function $f:\{0,1\}^n\rightarrow A$ is in $C^n(\mathcal{D}_k(A))$ if for every cube morphism $\phi:\{0,1\}^{k+1}\rightarrow\{0,1\}^n$ we have that
$$\sum_{v\in\{0,1\}^{k+1}}f(\phi(v))(-1)^{h(v)}=0$$
where $h(v)=\sum_{i=1}^{k+1}v_i$.
It was shown in \cite{NP} that higher degree abelian groups are building blocks of $k$-step nilspaces.
To state the precise statement we will need the following formalism.

Let $A$ be an abelian group and $X$ be an arbitrary set. An $A$ bundle over $X$ is a set $B$ together with a free action of $A$ such that the orbits of $A$ are parametrized by the elements of $X$. This means that there is a projection map $\pi:B\rightarrow X$ such that every fibre is an $A$-orbit. 
The action of $a\in A$ on $x\in B$ is denoted by $x+a$. Note that if $x,y\in B$ are in the same $A$ orbit then it make sense to talk about the difference $x-y$ which is the unique element $a\in A$ with $y+a=x$. In other words the $A$ orbits can be regarded as affine copies of $A$.

A $k$-fold abelian bundle $X_k$ is a structure which is obtained from a one element set $X_0$ in $k$-steps in a way that in the $i$-th step we produce $X_i$ as an $A_i$ bundle over $X_{i-1}$.
The groups $A_i$ are the structure groups of the $k$-fold bundle. We call the spaces $X_i$ the $i$-th factors. 
If all the structure groups $A_i$ and spaces $X_i$ are compact and the actions continuous then the $k$-fold bundle admits a probability measure which is built up from the Haar measures of the structure groups. For details see \cite{NP}.

\begin{definition} Let $X_k$ be a $k$-fold abelian bundle with factors $\{X_i\}_{i=1}^k$ and structure groups $\{A_i\}_{i=1}^k$. Let $\pi_i$ denote the projection of $X_k$ to $X_i$.
Assume that $X_k$ admits a cubespace structure with cube sets $\{C^n(X_k)\}_{n=1}^\infty$.
We say that $X_k$ is a {\bf $k$-degree bundle} if it satisfies the following conditions
\begin{enumerate}
\item $X_{k-1}$ is a $k-1$ degree bundle.
\item Every function $f\in C^n(X_{k-1})$ can be lifted to $f'\in C^n(X_k)$ with $f'\circ\pi_{k-1}=f$.
\item If $f\in C^n(X_k)$ then the fibre of $\pi_{k-1}:C^n(X_k)\rightarrow C^n(X_{k-1})$ containing $f$ is  $$\{f+g|g\in C^n(\mathcal{D}_k(A_k))\}.$$ 
\end{enumerate} 
\end{definition}

The next theorem form \cite{NP} says that $k$-degree bundles are the same as $k$-step nilspaces.

\begin{theorem} Every $k$-degree bundle is a $k$-step nilspace and every $k$-step nilspace arises as a $k$-degree bundle. 
\end{theorem}

It will be important that if $N$ is a $k$-step nilspace then the set $C^n(N)$ admits a $k$-fold bundle structure and so if $N$ is compact then $C^n(N)$ has a natural probability measure on it.
We will need a generalization of this probability measure.
Let $C$ be an arbitrary cube space with a given subset $S\subset C$ and function $f:S\rightarrow N$. We define $\Hom_f(C,N)$ as the set of morphisms whose restrictions to $S$ is $f$.
In many cases we can define a natural probability measure on $\Hom_f(C,N)$. A more detailed discussion of this can be found in \cite{NP}.

Now we define a sequence of special cube spaces (called three-cubes) which turn out to be helpful in this topic (see also \cite{NP} where three-cubes play a crucial role.)
Let $T_n=\{-1,0,1\}^n$ together with the following cubespace structure.
For every $v\in\{0,1\}^n$ we define the injective map $\Phi_v:\{0,1\}^n\rightarrow T_n$ by
$$\Phi_v(w_1,w_2,\dots,w_i)_j=(1-2v_j)(1-w_j).$$
We consider the smallest cubespace structure on $T_n$ in which all the maps $\Phi_v$ are morphisms.
We will also need the embedding $\omega:\{0,1\}^n\rightarrow T_n$ defined by $\omega(v)=\Phi_v(0^n)$.
It was proved in \cite{NP} that if $\psi\in\Hom(T_n,N)$ then $\omega\circ\psi\in C^n(N)$.
Let $f:\omega(\{0,1\}^n)\rightarrow N$ be a function. Then, according to \cite{NP}, we have that $\Hom_f(T_n,N)$ has a natural probability space structure provided that $N$ is a $k$-step compact nilspace.
An important property of this probability space is that a random $\phi\in\Hom_f(T_n,N)$ composed with $\Phi_v$ is a random morphism of $\{0,1\}^n$ to $N$ with the restriction that the image of $0^n$ is $f(\omega(v))$.
We call $T_n$ the {\bf three-cube} of dimension $n$.

\bigskip

\subsection{Cocycles}

Let $N$ be a nilspace. We say that two cubes $f_1:\{0,1\}^k\rightarrow N$ and $f_2:\{0,1\}^k\rightarrow N$ in $C^k(N)$ are adjacent if they satisfy that $f_1(v,1)=f_2(v,0)$ for every $v\in\{0,1\}^{k-1}$. For such cubes we define their concatenation as the function $f_3:\{0,1\}^k\rightarrow N$ with $f_3(v,0)=f_1(v,0)$ and $f_3(v,1)=f_2(v,1)$ for every $v\in\{0,1\}^{k-1}$.
The nilspace axioms imply that $f_3\in C^k(N)$.

If $\sigma$ is an automorphism of the cube $\{0,1\}^k$ then the we define $s(\sigma):=(-1)^m$ where $m$ is the number of $1$'s in the vector $\sigma(0^k)$. The automorphism $\sigma$ also acts on $C^k(N)$ by composition.

\begin{definition}\label{cocycle} Let $N$ be a nilspace and $G$ be an abelian group. A cocycle of degree $k-1$ is a function $\rho:C^k(N)\rightarrow G$ with the following two properties. 
\begin{enumerate}
\item If $f\in C^k(N)$ and $\sigma\in{\rm aut}(\{0,1\}^k)$ then $\rho(\sigma(f))=s(\sigma)\rho(f)$.
\item If $f_3$ is the concatenation of two cubes $f_1,f_2\in C^k(N)$ then $\rho(f_3)=\rho(f_1)+\rho(f_2)$. 
\end{enumerate}
\end{definition}

\begin{remark} We will also work with cocycles whose values are naturally represented on the unit circle in $\mathbb{C}$ with multiplication as group operation. To avoid confusion, we will call such cocycles {\bf multiplicative cocycles}.
\end{remark}

From now on we will always assume that $N$ is a compact $n$-step nilspace and $G$ is a compact abelian group. We will only consider measurable cocycles on $N$. It is proved in \cite{NP} that a measurable cocycle defines a continuous extension of $N$ by $G$. We review the construction of this nilspace.

For every $x\in N$ we introduce the space $$C_x^k(N)=\Hom_{0^k\mapsto x}(\{0,1\}^k,N).$$
The spaces $C_x^k(N)$ are the fibres of the map $\zeta:C^k(N)\rightarrow N$ defined by $\zeta(f)=f(0^k)$.
Each space $C_x^k(N)$ is a $k$-fold abelian bundle (see \cite{NP}) and consequently has its own probability space structure. We denote the probability measure on $C_x^k(N)$ by $\mu_x$. The measures $\{\mu_x\}_{x\in N}$ are forming a continuous system of measures (CSM). This means that for every continuous function $h:C^k(N)\rightarrow\mathbb{C}$ the function $$x\mapsto\int_{y\in C_x^k(N)}h(y)~d\mu_x$$ is continuous on $N$.
A good reference for CSM's is \cite{AD}. For a compact probability space $X$ and compact space $Y$ we denote the set of $Y$-valued measurable functions (up to $0$ measure change) by $L(X,Y)$.
For us it will be important that the function spaces $L(C_x^k(N),G)$ are also connected in a continuous way. Let $\mathcal{L}_k(N,G)=\cup_x L(C_x^k(N),G)$. The projection $\pi:\mathcal{L}_k(N,G)\rightarrow N$ is defined by $\pi(f)=x$ if $f\in L(C_x^k(N),G)$. We define the topology on $\mathcal{L}_k(N,G)$ as the weakest topology in which the following functions are continuous:
\begin{equation}\label{topdef}
f\mapsto \int_{y\in C_{\pi(f)}^k(N)}F_1(f(y))F_2(y)~d\mu_{\pi(f)}
\end{equation}
where $F_1:G\rightarrow\mathbb{C}$ and $F_2:C^k(N)\rightarrow\mathbb{C}$ are continuous functions. With this topology $\mathcal{L}_k(N,G)$ becomes a Polish space.

Let $\rho:C^k(N)\rightarrow G$ be a measurable function. We denote by $\rho_x$ its restriction to $C^k_x(N)$.
For a function in $f\in L(C_x^k(N),G)$ we denote the function set $\{f+a|a\in G\}$ by $f+G$.
We have the following \cite{NP}.

\begin{proposition}\label{meascont} Let $\rho:C^k(N)\rightarrow G$ be a measurable cocycle of degree $k-1$. Then $$M=\bigcup_{x\in N}\{\rho_x+G\}\subset\mathcal{L}_k(N,G)$$ is a compact $G$ bundle over $N$ with projection $\pi$.  
\end{proposition}

Now we define cubes of dimension $k$ on the compact topological space $M$.
Let $f:\{0,1\}^k\rightarrow M$ be a function. We have for every $v\in\{0,1\}^k$ that $\rho_{\pi(f(v))}=f(v)+a(v)$ for some element $a(v)$ in $G$. We say that $f$ is in $C^k(M)$ if $f\circ\pi\in C^k(N)$ and
$$\sum_{v\in\{0,1\}^k}a(v)(-1)^{h(v)}=\rho(f\circ\pi).$$ 
In general $f$ is in $C^n(M)$ if $f\circ\pi\in C^n(N)$ and every $k$-dimensional face restriction of $f$ is in $C^k(M)$.
Easy calculation shows that this defines a continuous nilspace structure on $M$.

It will be useful to specify a system of continuous functions $f_i:M\rightarrow\mathbb{C}$ which generate the topology on $M$.
For $v\in\{0,1\}^k$ we denote by $\psi_v:C^k(N)\rightarrow N$ the function with $\psi_v(c)=c(v)$ where $c:\{0,1\}^k\rightarrow N$ is in $C^k(N)$.  
Let $\mathcal{Q}$ denote the set of functions on $C^k(N)$ of the form $\prod_{v\in\{0,1\}^k}\psi_v\circ g_v$ where $\{g_v\}_{v\in\{0,1\}^k}$ is a system of continuous functions on $N$. It is clear that $\mathcal{Q}$ is a separating system of continuous functions closed under multiplication. By the Stone-Weierstrass theorem every continuous functions on $C^k(N)$ can be approximated by a finite linear combination of functions from $\mathcal{Q}$ in $L^\infty$.
The linear characters on $G$ are also forming a separating set of functions closed under multiplication.
We obtain the next lemma.

\begin{lemma}\label{topgen} The functions in (\ref{topdef}) where $F_1$ is a linear character of $G$ and $F_2$ is in $\mathcal{Q}$ generate the topology on $M$.
\end{lemma}

\subsection{Ultra product spaces}

Let $\omega$ be a non principal ultra filter on the natural numbers. 
Let $\{X_i,B_i,\mu_i\}_{i=1}^\infty$ be triples where $X_i$ is a compact Hausdorff space, $B_i$ is the Borel $\sigma$-algebra on $X_i$ and $\mu_i$ is a Borel probability measure on $B_i$.
We denote by $\bX$ the ultra product space $\prod_{\omega}X_i$.
The space $\bX$ has the following important structures on it.

\medskip

\noindent{\bf Strongly open sets} ~We call a subset of $\bX$ strongly open if it is the ultra product of open sets in $X_i$.

\medskip

\noindent{\bf Open sets:}~We say that $S\subset \bX$ is open if it is a countable union of strongly open sets. Open sets on $\bX$ give a $\sigma$-topology. This is similar to a topology but only countable unions of open sets are assumed to be open. Finite intersections and countable unions of open sets are again open. It can be proved that $\bX$ with this $\sigma$-topology is countably compact. If $\bX$ is covered by countably many open sets then there is a finite sub-system which covers $\bX$.

\medskip

\noindent{\bf Borel sets:} A subset of $\bX$ is called Borel if it is in the $\sigma$-algebra generated by strongly open sets. We denote by $\mathcal{A}(\bX)$ the $\sigma$ algebra of Borel sets.

\medskip

\noindent{\bf Ultra limit measure:}  If $S\subseteq \bX$ is a strongly open set of the form $S=\prod_\omega S_i$ then we define $\mu(S)$ as $\lim_\omega\mu_i(S_i)$. It is well known that $\mu$ extends as a probability measure to the $\sigma$-algebra of Borel sets on $\bX$. 

\medskip

\noindent{\bf Continuity:} A function $f:X\rightarrow T$ from $\bX$ to a topological space $T$ is called continuous if $f^{-1}(U)$ is open in $\bX$ for every open set in $T$. If $T$ is a compact Hausdorff topological space then $f$ is continuous if and only if it is the ultra limit of continuous functions $f_i:X_i\rightarrow T$. Furthermore the image of $X$ in a compact Hausdorff space $T$ under a continuous map is compact.

The fact that an ultra product space is endowed only with a $\sigma$-topology and not a topology might be upsetting for the first look. However one can look at $\bX$ as a space which is glued together form many ``ordinary'' topological spaces. Indeed, if we choose countably many open sets in $\bX$ then their finite intersections are forming a basis for a compact topological space. We will call such spaces {\bf separable topological factors} of $\bX$.

A natural way to define separable topological factors uses collections of continuous functions of the form $f:\bX\rightarrow\mathbb{C}$. If $\mathcal{F}$ is a set of continuous functions on $X$ then there is a smallest $\sigma$-topology in which all of them are continuous. This space will be called the $\sigma$-topology generated by $\mathcal{F}$. It turns out that if $\mathcal{F}$ is countable then they generate a separable topological factor.
This follows from the fact that the same $\sigma$-topology is generated by the pre-images of open balls in $\mathbb{C}$ with rational center and rational radius. The topological factor generated by a countable system of continuous functions is a compact second countable Hausdorff space. 

\subsection{Corner convolution}

In this chapter $A$ is either a compact Hausdorff group with the Haar measure or the ultra product of such groups.
If $f:A\rightarrow\mathbb{C}$ then we define the function $\Delta_t f$ by $(\Delta_t f)(x)=f(x)\overline{f(x+t)}$.
The Gowers norm $\|f\|_{U_k}$ is defined by
$$\|f\|_{U_k}^{2^k}=\mathbb{E}_{x,t_1,t_2,\dots,t_k}\Delta_{t_1,t_2,\dots,t_k}f(x)$$
for $f\in L^\infty(A)$.

Let $F=\{f_S\}_{S\subseteq [k]}$ be a system of $L^\infty$ functions on $A$.
The Gowers inner product of $F$ is defined by 
$$(F)=\mathbb{E}_{x,t_1,t_2,\dots,t_k}\prod_{S\subseteq [k]}f_S^{\epsilon(S)}(x+\sum_{i\in S}t_i)$$
where $\epsilon(S)$ is the conjugation if $|S|$ is odd and is the identity if $|S|$ is even.

The so-called Gowers-Cauchy-Schwartz inequality says that with the above notation
\begin{equation}\label{GCS}
(F)\leq\prod_{S\subseteq [k]}\|f_S\|_{U_k}.
\end{equation}

Let $K_n=2^{[n]}\setminus\{\emptyset\}$ denote the collection of non-empty subsets of $[n]=\{1,2,\dots,n\}$.
We can also identify $K_n$ with $\{0,1\}^n\setminus\{0\}$ where $0$ is the short hand notation for $0^n$.
For a function system $F=\{f_S\}_{S\in K_n}$ with $f_S\in L_\infty(A)$ we define the one variable function
\begin{equation}\label{cornerconv}
\mathcal{K}_n(F)(x)=\mathbb{E}_{t_1,t_2,\dots,t_n}\Bigl(\prod_{S\in K_n}f^{\epsilon(S)}_S(x+\sum_{i\in S}t_i)\Bigr).
\end{equation}

\begin{lemma} If $k\geq 2$ then $\|f\|_{U_k}\leq\|f\|_{2^{k-1}}$.
\end{lemma}

\begin{proof} We prove the statement by induction. If $k=2$ then $$\|f\|_{U_2}^4=\int_{t_1}\Bigl|\int_x f(x)\overline{f(x+t_1)}\Bigr|^2\leq\|f\|_2^4.$$
Let $f_t$ denote the function with $f_t(x)=f(x+t)$. We have by induction that
$$\|f\|_{U_{k+1}}^{2^{k+1}}=\int_{t}\|f\overline{f_t}\|_{U_k}^{2^k}\leq\int_t\|f\overline{f_t}\|_{2^{k-1}}^{2^k}=\int_{t,x}|f(x)|^{2^k}|f(x+t)|^{2^k}=\|f\|_{2^k}^{2^{k+1}}.$$
\end{proof}

\begin{corollary}\label{l2becs} If $k\geq 2$ and $|f|\leq 1$ then $(f,f)=\|f\|_2^2\geq \|f\|_{U_k}^{2^{k-1}}$.
\end{corollary}

\begin{proof} If $|f|\leq 1$ then $\|f\|_{U_k}^{2^{k-1}}\leq\|f\|_{2^{k-1}}^{2^{k-1}}\leq\|f\|_2^2$.

\end{proof}

\begin{lemma}\label{cornineq} For every $j\in[n]$ we have that
$$|\mathcal{K}_n(F)(x)|\leq\prod_{S\in K_n,j\notin S}\|f_S\|_{2^{n-1}}\prod_{S\in K_n,j\in S}\|f_S\|_{U_n}.$$
\end{lemma}

\begin{proof} If $n=1$ then the statement is true with equality. If $n>1$ then by induction we assume that it is true for $n-1$. Without loss of generality (using symmetry) we can assume that $j\neq n$. For every $S\in K_{n-1}$ we denote by $f_{S,t}$ the function $y\mapsto\overline{f}_{S\cup\{n\}}(y+t)f_S(y)$. Let $F_t=\{f_{S,t}\}_{S\in K_{n-1}}$. Then 
$$\mathcal{K}_n(F)(x)=\mathbb{E}_t(\mathcal{K}_{n-1}(F_t)(x))$$
and so by induction
$$|\mathcal{K}_n(F)(x)|\leq\mathbb{E}_t\Bigl(\prod_{S\in K_{n-1},j\notin S}\|f_{S,t}\|_{2^{n-2}}\prod_{S\in K_{n-1},j\in S}\|f_{S,t}\|_{U_{n-1}}\Bigr)\leq$$

$$\prod_{S\in K_{n-1},j\notin S}\mathbb{E}_t(\|f_{S,t}\|_{2^{n-2}}^{2^{n-1}})^{2^{-(n-1)}}\prod_{S\in K_{n-1},j\in S}\mathbb{E}_t(\|f_{S,t}\|_{U_{n-1}}^{2^{n-1}})^{2^{-(n-1)}}.$$
On the other hand we claim that
\begin{equation}\label{comp1}
\mathbb{E}_t(\|f_{S,t}\|_{2^{n-2}}^{2^{n-1}})^{2^{-(n-1)}}\leq\|f_S\|_{2^{n-1}}\|f_{S\cup\{n\}}\|_{2^{n-1}}
\end{equation}
and
\begin{equation}\label{comp2}
\mathbb{E}_t(\|f_{S,t}\|_{U_{n-1}}^{2^{n-1}})^{2^{-(n-1)}}\leq\|f_S\|_{U_n}\|f_{S\cup\{n\}}\|_{U_n}.
\end{equation}

Inequality (\ref{comp1}) follows by
$$\mathbb{E}_t(\|f_{S,t}\|_{2^{n-2}}^{2^{n-1}})=\mathbb{E}_t\Bigl(\mathbb{E}^2_y(|f_S(y)|^{2^{n-2}}|f_{S\cup\{n\}}(y+t)|^{2^{n-2}})\Bigr)\leq$$
$$\leq\mathbb{E}_{t,y}\Bigl(|f_S(y)|^{2^{n-1}}|f_{S\cup\{n\}}(y+t)|^{2^{n-1}}\Bigr)=\|f_S\|_{2^{n-1}}^{2^{n-1}}\|f_{S\cup\{n\}}\|_{2^{n-1}}^{2^{n-1}}.$$

To see (\ref{comp2}) let $G=\{g_S\}_{H\subseteq[n]}$ be the function system defined by $g_H=f_S$ if $H\subseteq[n-1]$ and $g_H=f_{S\cup\{n\}}$ if $n\in H$. Then by (\ref{GCS})
$$\mathbb{E}_t(\|f_{S,t}\|_{U_{n-1}}^{2^{n-1}})=(F)\leq\|f_S\|_{U_n}^{2^{n-1}}\|f_{S\cup\{n\}}\|_{U_n}^{2^{n-1}}$$
which completes the proof.
\end{proof}

\medskip

\begin{corollary}\label{cornineqcor} If $F=\{f_S\}_{S\in K_n}$ is an $L^\infty$ function system on an ultra product group $\bA$ and there exists an $S\in K_n$ with $\|f_S\|_{U_n}=0$ then $\mathcal{K}_n(F)(x)=0$ for every $x\in \bA$. Furthermore for an arbitrary function system $F=\{f_s\}_{S\in K_n}$ we have that if $F'=\{\mathbb{E}(f_S|\mathcal{F}_{n-1})\}_{S\in K_n}$ then $\mathcal{K}_n(F)(x)=\mathcal{K}_n(F')(x)$ holds for every $x\in\bA$.
\end{corollary}

\begin{proof} The first part is a direct consequence of lemma \ref{cornineq}.
The second part follows from the first part using a $2^n-1$ step process in which we modify the functions $f_S$ to $\mathbb{E}(f_S|\mathcal{F}_{n-1})$ one in each step. The difference $f_S-\mathbb{E}(f_S|\mathcal{F}_{n-1})$ has zero $U_n$-norm and so by the multi linearity of $\mathcal{K}_n$ the steps of the process don't modify $\mathcal{K}_n$.
\end{proof}

\begin{lemma} Let $\bA=\prod_\omega A_i$ be the ultra product of compact abelian groups. Then $\mathcal{K}_n(F)$ is continuous for every system $F$ of $L^\infty$ functions.
\end{lemma}

\begin{proof} We get from lemma \ref{cornineq} that changing any of the functions in the system $F$ on a $0$-measure set does not change any value of $\mathcal{K}(F)$.
This means that if $A=\prod_\omega A_i$ is an ultra product group then without loss of generality we can assume that every function in the system $F$ is the ultra limit of continuous functions on the compact groups $A_i$. It follows that the function $\mathcal{K}(F)$ is the ultra limit of functions $\mathcal{K}(F_i)$ where $F_i$ is a system of continuous functions on the compact abelian group $A_i$. This completes the proof.
\end{proof}

\subsection{Higher order Fourier analysis on ultra product groups}

\bigskip

In this section we summarize some results from the papers \cite{Sz1},\cite{Sz2}.
Let $\{A_i\}_{i=1}^\infty$ be a sequence of compact abelian groups and let $\bA$ denote their ultra product according to a fixed ultra filter $\omega$. The ultra product group $\bA$ is endowed with the $\sigma$-algebra $\mathcal{A}$ generated by ultra products of open sets.
We define the measure $\mu$ on $\mathcal{A}$ as the ultra limit of the Haar measures in the sequence $\{A_i\}_{i=1}^\infty$.
Gowers norms can be defined on $\bA$ but they are only semi norms on $L^\infty(\bA,\mathcal{A})$.
It was proved in \cite{Sz1} that for every natural number $k$ there is a unique maximal sub $\sigma$-algebra $\mathcal{F}_k$ in $\mathcal{A}$ such that $U_{k+1}$ is a norm on $L^\infty(\bA,\mathcal{F}_k)$. Furthermore if $f\in L^\infty(\bA,\mathcal{A})$ then $\|f\|_{U_{k+1}}=\|\mathbb{E}(f|\mathcal{F}_k)\|_{U_{k+1}}$. In other words every bounded measurable function $f$ on $\bA$ is (uniquely) decomposable as $f=f_k+g$ such that $f_k$ is measurable in $\mathcal{F}_k$ and $\|g\|_{U_{k+1}}=0$. We can look at this decomposition in a way that $f_k=\mathbb{E}(f|\mathcal{F}_k)$ is the structured part of $f$ and $g$ is the random part. This notion of randomness depends on the number $k$. 
Since $\|f\|_{U_k}\leq\|f\|_{U_{k+1}}$ for every $k$ we have that $\mathcal{F}_1\subset\mathcal{F}_2\subset\mathcal{F}_3\subset\dots$.
(We can also introduce $\mathcal{F}_0$ as the trivial $\sigma$-algebra on $\bA$.)

The $\sigma$-algebras $\mathcal{F}_k$ have many equivalent descriptions. The simplest uses $k+1$ dimensional cubes.
The elements in the group $\bA^{k+2}$ with coordinates $(x,t_1,t_2,\dots,t_k)$ are in a one to one correspondence with the $k+1$ dimensional cubes $f:\{0,1\}^{k+1}\rightarrow \bA$ such that $f(v)=x+\sum_{i=1}^{k+1}v_it_i$. Wit this notation, every vertex $v\in\{0,1\}^k$ gives rise to the homomorphism $\psi_v:\bA^{k+2}\rightarrow\bA$ defined by $\psi_v(f)=f(v)$.
Let $\mathcal{A}_v=\psi_v^{-1}(\mathcal{A})$ for $v\in\{0,1\}^{k+1}$. Then we have that
\begin{equation}\label{fk}
\psi_0^{-1}(\mathcal{F}_k)=\mathcal{A}_0\bigcap\bigl(\bigcup_{v\in\{0,1\}^{k+1}\setminus\{0\}}\mathcal{A}_v\bigr).
\end{equation}
Note that this formula does not fully use the abelian group structure on $\bA$. The same formula makes sense on the ultra products of compact nilspaces $\{N_i\}_{i=1}^\infty$. In that case $\bA^{k+2}=C^{k+1}(\bA)$ is replaced by the ultra product of the spaces $C^{k+1}(N_i)$.

It is interesting to mention that there is another description of $\mathcal{F}_k$. The elements of $\mathcal{F}_1$ are those measurable sets $S\subset\bA$ whose shifts $S+x$ generate a separable $\sigma$-algebra as $x$ runs through all elements in $\bA$.
We introduce $\mathcal{F}_k$ recursively. A set $S\subset\bA$ is measurable in $\mathcal{F}_k$ if the $\sigma$-algebra generated by the shifts $\{S+x|x\in\bA\}$ and $\mathcal{F}_{k-1}$ can also be generated by $\mathcal{F}_{k-1}$ and countable many measurable sets $\{S_i\}_{i=1}^\infty$. 

Ordinary Fourier analysis is related to the fact the $L^2(\bA,\mathcal{F}_1)$ is the orthogonal sum (in a unique way) of one dimensional shift invariant subspaces. These subspaces are generated by continuous characters $\chi:\bA\rightarrow\mathbb{C}$. Recall that a character is a homomorphism to the complex unit circle (as a group with multiplication). 
This decomposition has a generalization to $\mathcal{F}_k$. The space $L_2(\bA,\mathcal{F}_k)$ is a module over the algebra $L^\infty(\bA,\mathcal{F}_{k-1})$. It turns out \cite{Sz1},\cite{Sz2} that there is a unique decomposition
\begin{equation}\label{hdecomp1}
L^2(\bA,\mathcal{F}_k)=\bigoplus_{\phi\in\hat{\bA}_k} V_\phi
\end{equation}
where $V_\phi$ is a shift invariant rank one module over $L^\infty(\bA,\mathcal{F}_{k-1})$.
This means that in each space $V_\phi$ there is a function $\chi:\bA\rightarrow\mathbb{C}$ with $|\chi|=1$ such that
the set $\chi\cdot L^\infty(\bA,\mathcal{F}_{k-1})$ is dense in $V_\phi$ in the $L^2$ metric. Such functions will be called $k$-th order characters. We will denote the set of $k$-th order characters in $V_\phi$ by $V^*_\phi$.

It is a crucial fact that point wise multiplication of elements from these modules defines a group structure on $\hat{\bA}_k$.
We say that $\hat{\bA}_k$ is the $k$-th order dual group of $\bA$.
The decomposition (\ref{hdecomp1}) implies that for every function $f\in L^2(\bA,\mathcal{A})$ there is a unique decomposition
$$f=g+\sum_{\phi\in\hat{\bA}_k}f_\phi$$ converging in $L^2$ where $g=f-\mathbb{E}(f|\mathcal{F}_k)$ and $f_\phi$ is the projection of $f$ to $V_\phi$. In particular if $f\in L^2(\bA,\mathcal{F}_k)$ then $g=0$. 
We call this the $k$-th order Fourier decomposition of $f$. Note that the $k$-th order Fourier decomposition has only countable non zero terms.

Let $V^\infty_\phi$ denote the set of bounded functions in $V_\phi$. For a function $f:\bA\rightarrow\mathbb{C}$ we introduce the function
$$(f)^\diamond=\prod_{v\in\{0,1\}^{k+1}}\psi_v\circ f^{\epsilon(v)}$$
on $C^{k+1}(\bA)$. For $x\in\bA$ we denote by $(f)^\diamond_x$ the restriction of $(f)^\diamond$ to $C^{k+1}_x(\bA)$ divided by $f(x)$.

The next lemma \cite{Sz2} is crucial for the results in this paper.

\begin{lemma}\label{kisebb} Let $\phi\in\hat{\bA}_k$ and $f\in V^\infty_\phi$. Then the function $(f)^\diamond$ is measurable in
$$\bigcup_{v\in\{0,1\}^{k+1}}\psi_v^{-1}(\mathcal{F}_{k-1})$$
on $C^{k+1}(\bA)$.
\end{lemma}
In fact this can be used as an alternative definition for elements in rank one modules.
(An advantage of this definition is that it does not use shifts and so it can be generalized to nilspaces.)

\bigskip

In the paper \cite{Sz2} there is detailed analysis of the structures of $\hat{\bA}_k$.
Without going into the details of that we mention three important facts.

\begin{proposition} We have the following statements for the structure of $\hat{\bA}_k$.

\begin{enumerate}
\item $\bA\simeq\hat{\bA}_1$.
\item Let $n$ be a fixed natural number. If all the groups $A_i$ are of exponent $n$, then also $\hat{\bA}_k$ has exponent $n$ for every $k$. 
\item If the smallest prime divisor of $|A_i|$ goes to infinity with $i$ then $\hat{\bA}_k$ is torsion free for every $k$.
\item If all the groups $A_i$ are generated by at most $d$ elements then $\hat{\bA}_k$ is torsion free for every $k\geq 2$.
\end{enumerate}
\end{proposition}

\bigskip

\subsection{Nilspace factors of ultra product groups}

\bigskip

\begin{definition} A $k$-step nilspace factor of $\bA$ is given by a continuous morphism $\Psi:\bA\rightarrow N$ into a compact $k$-step nilspace such that the induced maps $C^n(\bA)\rightarrow C^n(N)$ are surjective for every $n$.
Equivalently, a continuous surjective function $\Psi:\bA\rightarrow T$ into a compact Hausdorff space $T$ defines a nilspace factor if the cubespace structure on $T$ (obtained by composing cubes in $\bA$ with $f$) is a $k$-step nilspace. 
\end{definition}

We can also think of a nilspace factor as an equivalence relation on $\bA$ whose classes are the fibres of $\Psi$.
Let $\Psi:\bA\rightarrow N$ be any continuous morphism of $\bA$ into a compact $k$-step nilspace. It follows from equation (\ref{fk}) and the unique closing property that $\Psi$ is measurable in $\mathcal{F}_k$. 
Assume that the structure groups of $N$ are $A_1,A_2,\dots,A_k$.
If $\chi\in\hat{A}_i$ is a continuous linear character on $A_i$ then we can represent it by a Borel function $\chi':N_i\rightarrow\mathbb{C}$ on the $i$-step factor $N_i$ of $N$ such that $\chi'(x+a)=\chi'(x)\chi(a)$ and $|\chi'|=1$ everywhere.
We have that $\Psi\circ\pi_i\circ\chi'$ is a $k$-th order character on $\bA$ and its module does not depend on the choice of $\chi'$ (only on $\chi$). This induces homomorphisms $\tau_i:\hat{A_i}\rightarrow\hat{\bA}_i$.

\begin{definition} We say that $\Psi$ is {\bf character preserving} if all the homomorphisms $\tau_i$ are injective.
\end{definition}

We will say that a nilspace factor $\Psi:\bA\rightarrow N$ is {\bf measure preserving} if the induced maps $C^n(\bA)\rightarrow C^n(N)$ are all measure preserving. We say that $\Psi$ is {\bf rooted measure preserving} if for every $a\in\bA$ and natural number $n\in\mathbb{N}$ the map $C^n_a(\bA)\rightarrow C^n_{\Psi(a)}(N)$ induced by $\Psi$ is measure preserving.

\begin{theorem}\label{charpres} Let $\Psi:\bA\rightarrow N$ be a $k$-step nilspace factor. Then the following statements are equivalent.
\begin{enumerate}
\item $\Psi$ is character preserving,
\item $\Psi$ is rooted measure preserving,
\item $\Psi$ is measure preserving.
\end{enumerate}
\end{theorem} 

Before proving this theorem we need some preparation.
Let $K_n$ denote the set $\{0,1\}^n\setminus\{0^n\}$.
For an abelian group $A$ we denote by $Z_{n,k}(A)$ the set of maps $m:\{0,1\}^n\rightarrow A$ such that for every $n-k$ dimensional face $F\subseteq\{0,1\}^n$ the equality $\sum_{v\in F}m(v)=0$ holds.
Similarly we denote by $Z_{n,k}^*$ the set of maps $m:\{0,1\}^n\rightarrow A$ such that $\sum_{v\in F}m(v)=0$ holds for every $n-k$ dimensional face contained in $K_n$. It is easy to see that the homomorphism $Z_{n,k}(A)\rightarrow Z_{n,k}^*(A)$ given by forgetting to coordinate at $0^k$ is surjective.

It was proved in \cite{Sz3} that the abelian group $Z_{n,k}(A)$ is generated by the elements $g_{F,a}$ where $F$ is a $k+1$ dimensional face, $a\in A$ and $g_{F,a}(v)=a(-1)^{h(v)}$ if $v\in F$ and $g_{F,a}(v)=0$ elsewhere.
This implies that if $G$ is a compact abelian group then $C^n(\mathcal{D}_k(G))\subset G^{\{0,1\}^n}$ is equal to the kernel of $Z_{n,k}(\hat{G})$ where $\hat{G}$ is the dual group of $G$. 

We will consider the embedding $\tau$ of $G^{K_n}$ into $G^{\{0,1\}^n}$ where the coordinate at $0^k$ is set to be $0$.
This embedding induces a homomorphism $\hat{\tau}$ from $\hat{G}^{\{0,1\}^n}$ to $\hat{G}^{K_n}$. The homomorphism $\hat{\tau}$ is given by forgetting the component at $0^n$. This means that the image of $Z_{n,k}(\hat{G})$ under $\hat{\tau}$ is $Z_{n,k}^*(\hat{G})$. We obtain that $C^n_0(\mathcal{D}_k(G))$ is equal to the kernel of $Z_{n,k}^*(\hat{G})$.

\bigskip

Let $\psi_v:C^n_x(\bA)\rightarrow\bA$ denote the map with $\psi_v(f)=f(v)$.
\begin{lemma}\label{charconv} Let $t:K_n\rightarrow\hat{\bA}_k$ be a map such that $t\notin Z_{n,k}^*(\hat{\bA}_k)$ and let $\{f_v\}_{v\in K_n}$ be a system of functions on $\bA$ such that $f_v$ is in $V_{t(v)}^\infty$. Then 
$$\int_{C^n_0(\bA)}\prod_{v\in K_n}\psi_v\circ f_v~d\mu=0.$$
\end{lemma} 

\begin{proof} Let $F\subset K_n$ be an $n-k$ dimensional face such that $h=\sum_{v\in F}t(v)\neq 0$. Without loss of generality (using the symmetries of $K_n$) we can assume that $F$ is obtained by changing the first $n-k$ coordinates in all possible ways in $\{0,1\}^n$ such that the last $k$-coordinates are given by a fix nonzero vector $w\in\{0,1\}^k$. The elements in $C^n_0(\bA)$ can be represented by vectors $(t_1,t_2,\dots,t_n)\in\bA^n$ and $\psi_v(t_1,t_2,\dots,t_n)=\sum_{i=1}^n t_iv_i$.
By the non-stand Fubini theorem \cite{ESz} we have that the integral in the lemma is equal to
$$\int_{t_1,t_2,\dots,t_{n-k}}\int_{t_{n-k+1},\dots,t_n}\prod_{v\in K_n}\psi_v\circ f_v~d\mu$$
and so it is enough to show that for every fixed $t_1,t_2,\dots,t_{n-k}$ the inner integral is $0$.
For every $u\in\{0,1\}^k$ let $$g_u(x)=\prod_{v\in\{0,1\}^{n-k}}f_{(v,u)}\bigl(x+\sum_{i=1}^{n-k}v_it_i\bigr)$$ 
and let $G$ denote the function system $\{g_u^{\epsilon(u)}\}_{u\in K_k}$.
We have that (for fixed $t_1,t_2,\dots,t_{n-k}$)
$$\int_{t_{n-k+1},\dots,t_n}\prod_{v\in K_n}\psi_v\circ f_v~d\mu=g_0\mathcal{K}_k(G)(0).$$
Our condition implies that $g_w$ is contained in the nontrivial module $V_h^\infty$ and so $\|g_w\|_{U_k}=0$. It follows from corollary \ref{cornineqcor} that $\mathcal{K}_k(G)(0)=0$ which completes the proof.
\end{proof}

\bigskip

\noindent{\it Proof of theorem \ref{charpres}:}

We start with $(1)$ implies $(2)$.

For a character $\chi\in\hat{A_k}$ let $V_\chi^\infty(N)$ denote the set of functions $g$ in $L^\infty(N)$ satisfying $g(x+a)=g(x)\chi(a)$ for every $n\in N$ and $a\in A_k$. It is clear that $\Psi\circ g$ is in $V^\infty_{\tau_k(\chi)}$ for such a function $g$.

We proceed by induction on $k$. If $k=0$ then $N$ is a one point structure and there is nothing to prove. Assume that the statement holds for $k-1$. This means that $\Psi$ composed with the projection $\pi_{k-1}$ from $N$ to the $k-1$ step factor $N_{k-1}$ is a rooted measure preserving factor of $\bA$.

Using continuity we get that the image of $C^k(\bA)$ in $C^k(N)$ under $\Psi$ is a compact subset. The same statement holds for the spaces $C^k_x(\bA)$. This means that the measure preserving property of these maps implies surjectivity automatically.
It remains to show that $\Psi$ is rooted measure preserving.

By applying an appropriate translation in $\bA$ it is enough to prove the rooted measure preserving property in the case $x=0$. Assume that $\Psi(0)=y\in N$.
Let $\nu$ denote the probability distribution on $C_y^n(N)$ obtained by composing the uniform distribution on $C_0^n(\bA)$ by $\Psi$. Note that $C_y^n(N)$ is a $C^n_0(\mathcal{D}_k(A_k))$-bundle over $C^n_{\pi_{k-1}(y)}(N_{k-1})$. For this reason it is enough to show that $\nu$ is invariant under the natural action of $C_0^n(\mathcal{D}_k(A_k))$ on $C_y^n(N)$.
This invariance can be proved by showing for a function system $U$ which linearly spans an $L^1$-dense set in $L^\infty(C_y^n(N))$ that for every $r\in C_0^n(\mathcal{D}_k(A_k))$ and $u\in U$ the equation $\int u~d\nu=\int u^r~d\nu$ holds. (The shift $u^r$ of $u$ is defined as the function satisfying $u^r(y)=u(y+r)$.)

We will set
$$U=\{\prod_{v\in K_n}c_v\circ f_v~|~f_v\in V^\infty_{\chi_v}(N),\chi_v\in\hat{A_k},~~v\in K_n\}$$
where $c_v:C^n_y\rightarrow N$ is the map defined by $c_v(f)=f_v$.
The set $U$ is closed under multiplication and contains a separating system of continuous functions. It follows from the Stone-Weierstrass theorem that every function in $L^\infty(C_y^n(N))$ can be approximated by some finite linear combination of elements from $U$.

Let $\{f_v\}_{v\in K_v}$ be a function system with $f_v\in V^\infty_{\chi_v}$ for some character system $\{\chi_v\}_{v\in K_n}$.
We will denote by $t:K_n\rightarrow\hat{A_k}$ the function with $t(v)=\chi_v$. 
Let $f=\prod_{v\in K_n}c_v\circ f_v$. Let $r\in C_0^n(\mathcal{D}_k(A_k))$. We have that $f^r=\prod_{v\in K_n}c_v\circ f_v^{r_v}$ where $r_v$ is the component of $r$ at $v$. This means that 
\begin{equation}\label{chmult}
f^r=f\prod_{v\in K_n}\chi_v(r_v).
\end{equation}
There are two possibilities.
In the first case $t\notin Z_{n,k}^*(\hat{A_k})$. In this case by the character preserving property of $\Psi$ and by lemma \ref{charconv} we get that $\int f~d\nu=0$ and so by (\ref{chmult}) we have $\int f~d\nu=\int f^r~d\nu$.
In the second case $t\in Z_{n,k}^*(\hat{A_k})$. Then $r\in{\rm ker}(t)$ and so we have by (\ref{chmult}) that $f=f^r$.

\medskip

To see that $(2)$ implies $(3)$ notice that a random element in $C_x^{n+1}(\bA)$ (resp. $C_{\Psi(x)}^{n+1}(N)$) restricted to an $n$ dimensional face of $\{0,1\}^{n+1}$ not containing $0^{n+1}$ is $C^n(\bA)$ (resp. $C^n(N)$) with the uniform distribution. 

We finish with $(3)$ implies $(1)$. This follows from the fact that the measure preserving property guarantees that $\Psi$ preserves all the Gowers norms. Consequently the maps $\tau_i$ have all trivial kernels.

\subsection{Nil-$\sigma$-algebras}

For a $\sigma$-algebra $\mathcal{B}\subset\mathcal{A}$ we denote by $[\mathcal{B},k]^*$ the set of functions $f:\bA\rightarrow\mathbb{C}$ such that $|f|=1$ and $(f)^\diamond$ is measurable in $$\bigcup_{v\in\{0,1\}^{k+1}}\psi^{-1}_v(\mathcal{B}_{k-1}).$$

\bigskip

\begin{definition}\label{nilsigma} Let $\{\hat{A}_i\}_{i=1}^k$ be a sequence of countable abelian groups such that $\hat{A}_i\subset\hat{\bA}_i$ for $1\leq i\leq k$. For every $1\leq i\leq k$ let $J_i=\{\chi_\phi\}_{\phi\in \hat{A}_i}$ be a system of $i$-degree characters with $\chi_\phi\in V^*_\phi$ and let $\mathcal{B}_i$ denote the $\sigma$-algebra generated by the functions in $\cup_{j=1}^i J_j$ ($\mathcal{B}_0$ is defined as the trivial $\sigma$-algebra.)
We say that $\mathcal{B}=\mathcal{B}_k$ is a degree-$k$ nil-$\sigma$-algebra if for every $1\leq i\leq k$ the following two conditions hold.
\begin{enumerate}
\item $\chi_\phi\in [\mathcal{B}_{i-1},i]^*$ for every $\phi\in \hat{A}_i$,
\item $\mathcal{B}_i\cap\mathcal{F}_{i-1}=\mathcal{B}_{i-1}$.
\end{enumerate}
\end{definition}  

We will need the next lemma.

\begin{lemma}\label{sepnil} Every separable $\sigma$-algebra $\mathcal{C}$ in $\mathcal{F}_k$ is contained in degree $k$ nil-$\sigma$-algebra.
\end{lemma}

\begin{proof} The proof is an induction on $k$. Assume that the statement holds for $k-1$. 
Let $\hat{A}_k\subset\hat{\bA}_k$ denote the subgroup generated by all the modules $V_\phi$ that are not orthogonal to $L^2(\mathcal{C})$. Since $\mathcal{C}$ is separable and the modules are pairwise orthogonal we get that $\hat{A}_k$ is a countable group. Let us choose a representative $\chi_\phi$ from each $V^*_\phi$ with $\phi\in \hat{A}_k$ and assume that $\chi_\phi\in[\mathcal{B}_\phi,k]^*$ for some separable $\sigma$-algebra $\mathcal{B}_\phi$ in $\mathcal{F}_{k-1}$. Let $\mathcal{C}'$ be the $\sigma$-algebra generated by $\mathcal{C}$, the functions $\{\chi_\phi\}_{\phi\in \hat{A}_k}$ and the $\sigma$-algebras $\{\mathcal{B}_\phi\}_{\phi\in \hat{A}_k}$. By induction we can embed $\mathcal{C}'\cap\mathcal{F}_{k-1}$ into a degree $k-1$ nil-$\sigma$-algebra $\mathcal{B}_{k-1}$. We define $\mathcal{B}_k$ as the $\sigma$-algebra generated by $\mathcal{B}_{k-1}$ and the system $\{\chi_\phi\}_{\phi\in \hat{A}_k}$.
It is clear that $\mathcal{B}_k$ is a $k$-degree nil-$\sigma$-algebra. We have to check that it contains $\mathcal{C}$. 
To see this let $f\in L^\infty(\mathcal{C})$ be a function with $k$-th order decomposition $f=\sum f_\phi$. If $f_\phi\neq 0$ then $\phi\in \hat{A}_k$ and $f_\phi\chi^{-1}_\phi\in\mathcal{F}_{k-1}\cap\mathcal{C}'$. It follows that $f_\phi\chi^{-1}_\phi\in\mathcal{B}_{k-1}$ and so $f_\phi\in\mathcal{B}_k$.

\end{proof}

\subsection{The main theorem for ultra product groups}

\bigskip

\begin{theorem}\label{nilnil} For every degree-$k$ nil-$\sigma$-algebra $\mathcal{B}$ there is a character preserving $k$-step nilfactor $\Psi:\bA\rightarrow N$ such that $\mathcal{B}$ (up to $0$ measure sets) is equal to the $\sigma$-algebra generated by $\Psi$.
\end{theorem}

As a consequence we get the next theorem.

\begin{theorem}\label{ultdecomp} Let $k$ be a fixed natural number. Then every function $f\in L^\infty(\bA)$ has a decomposition $f=f_s+f_r$ such that $\|f_r\|_{U_{k+1}}=0$ and $f_s=\Psi\circ g$ for some character preserving $k$-step nilfactor $\Psi:\bA\rightarrow N$ and measurable function $g:N\rightarrow\mathbb{C}$.
\end{theorem}

Note that the decomposition $f=f_n+f_r$ is unique up to a $0$ measure change.
We devote the rest of this chapter to the proof of these two theorems. 

\bigskip

\noindent{\it Proof of theorem \ref{nilnil}:}~~We show the theorem by induction on $k$. If $k=0$ then $\mathcal{B}$ is the trivial $\sigma$-algebra so the statement is trivial.
Assume that the statement holds for $k-1$. We can assume that $\mathcal{B}$ is generated by a $k-1$ step nil-$\sigma$-algebra $\mathcal{B}_{k-1}$ and a system of characters $\{\chi_\phi\}_{\phi\in \hat{G}}$ for some countable subgroup $\hat{G}\subset\hat{\bA}_k$ such that $\chi_\phi\in V^*_\phi$ and $\chi_\phi\in [\mathcal{B}_{k-1},k]^*$.

By induction we have that $\mathcal{B}_{k-1}$ is equal to the $\sigma$-algebra generated by a $k-1$ step character preserving nilfactor $\Psi_{k-1}:\bA\rightarrow N_{k-1}$.
By abusing the notation we will denote by $\Psi_{k-1}$ all the maps $C^n(\bA)\rightarrow C^n(N_{k-1})$ induced by $\Psi_{k-1}$.

\bigskip

\noindent{Claim I.}~~{\it for every $\chi_\phi$ there is a function $\rho_\phi$ with $|\rho_\phi|=1$ and $\|\chi_\phi-\rho_\phi\|_1=0$ such that the function $(\rho_\phi)^\diamond$ is equal to the composition of $\Psi_{k-1}:C^{k+1}(\bA)\rightarrow C^{k+1}(N_{k-1})$ with a degree $k$ (multiplicative) cocycle $\kappa_\phi:C^{k+1}(N_{k-1})\rightarrow\mathbb{C}$.}

\medskip

First notice that $(\chi_\phi)^\diamond$ is a (multiplicative) coboundary on $C^{k+1}(\bA)$. 
On the other hand since $\chi_\phi\in[\mathcal{B}_{k-1},k]^*$ it follows that there is a Borel function $g:C^{k+1}(N)\rightarrow\mathbb{C}$ with $|g|=1$ such that $\Psi_{k-1}:C^{k+1}(\bA)\rightarrow C^{k+1}(N_{k-1})$ composed with $g$ is almost everywhere equal to $(\chi_\phi)^\diamond$. The fact that $(\chi_\phi)^\diamond$ satisfies the cocycle axioms and the fact that $\Psi_{k-1}$ is a measure preserving nilspace factor together imply that $g$ satisfies the cocycle axioms for almost every pair of cubes. It is proved in \cite{NP} that an almost cocycle in this sense can be corrected to a precise cocycle with a $0$ measure change. So let $\kappa_\phi$ denote such a corrected precise cocycle and let $q$ denote the composition of $\Psi_{k-1}$ with $\kappa_\phi$. We have that $q$ and $(\chi_\phi)^\diamond$ differ on a $0$ measure set. Not that $q$ is a multiplicative cocycle on $C^{k+1}(\bA)$ since $\Psi_{k-1}$ preserves the cubic structure. To prove the claim it is enough to show that $q'=q^{-1}(\chi_\phi)^\diamond$ is the coboundary corresponding to a function $p:\bA\rightarrow\mathbb{C}$ which is almost everywhere $1$.
The method to find $p$ is to verify that for every fixed $x\in\bA$ the function $q'$ restricted to $C^{k+1}_x(\bA)$ takes a constant value $p(x)$ at almost every point of $C^{k+1}_x(\bA)$. Then we show that $(p)^\diamond=q'$.

Let $t$ be a random element in $\Hom_{1^{k+1}\rightarrow x}(T_{k+1},\bA)$ where $T_{k+1}$ is the three cube of dimension $k+1$.
We have that $\Phi_v\circ t$ is a random element in $C^{k+1}(\bA)$ if $v\neq 0^{k+1}$. Thus with probability one we have that 
\begin{equation}\label{oneeq}
\prod_{v\in K_{k+1}}q'(\Phi_v\circ t)^{\epsilon(v)}=1
\end{equation}
since $q'$ is almost everywhere $1$. On the other hand (using the cocycle properties of $q'$) the left side of (\ref{oneeq}) is equal to $q'(\omega\circ t)/q'(\Phi_{0^{k+1}}\circ t)$.
The values $q'(\omega\circ t)$ and $q'(\Phi_{0^{k+1}}\circ t)$ are independent and have the same distribution as $q'$ on $C_x^{k+1}(\bA)$. 
This is only possible if this distribution is concentrated on one value $p(x)$.

Let $c\in C^{k+1}(\bA)$ be an arbitrary element and let $t$ be a random morphism of $T_{k+1}$ into $\bA$ with the restriction that $\omega\circ t$ is equal to $c$. Then with probability one $q'(\Phi_v\circ t)=p(c(v))$ for every $v\in\{0,1\}^{k+1}$. On the other hand by the cocycle property of $q'$ we have that
$$q'(c)=\prod_{v\in\{0,1\}^{k+1}}q'(\Phi_v\circ t)^{\epsilon(v)}.$$ This implies that $(p)^\diamond=q'$ and so $q=(\rho_\phi)^\diamond$ where $\rho_\phi=\chi_\phi/p$. 

\bigskip

\noindent{Claim II.}~~{\it If $\rho_\phi$ represents an element of order $d$ in $\hat{\bA}_k$ then $\rho_\phi^d$ is constant on the fibres of $\Psi_{k-1}$. Furthermore there is a function $f$ with $|f|=1$ depending on the factor $\Psi_{k-1}$ such that $(\rho_\phi/f)^d=1$.}
\bigskip

Since $\rho_\phi^d$ is measurable in $\mathcal{F}_{k-1}$ we get by the nil-$\sigma$-algebra properties of $\mathcal{B}_k$ that it is also measurable in $\mathcal{B}_{k-1}$. By our induction this implies that there is a function $g:\bA\rightarrow\mathbb{C}$ with $|g|=1$ such that $g$ is constant on the fibres of $\Psi_{k-1}$ and $g=\rho_\phi^d$ almost surely.
We have that for every $x\in\bA$ the function $(\rho_\phi^d)^\diamond/(g)^\diamond$ is equal to $\rho_\phi(x)^d/g(x)$ on almost every point of the space $C_x^{k+1}(\bA)$. This means by the rooted measure preserving property of $\Psi_{k-1}$ and the fact that both $(\rho_\phi^d)^\diamond$ and $(g)^\diamond$ depend on the factor $\Psi_{k-1}$ that $\rho_\phi(x)^d/g(x)$ depends only on the factor $\Psi_{k-1}$. This shows the first part of the claim.
It is easy to see $\rho_\phi^d$ has a measurable $d$-th root $f$ which depends on the factor $\Psi_{k-1}$ since $\rho_\phi^d$ depends on $\Psi_{k-1}$. 

\bigskip

To formulate the next claim we introduce a notation. For $\phi\in\hat{G}$ we denote by $Q_\phi$ the set of functions that can be obtained as $\mathcal{K}_{k+1}(F)$ where $F=\{f_v\rho_\phi\}_{v\in K_{k+1}}$ is a function system such that $f_v$ is continuous in the factor $\Psi_{k-1}$ for every $v\in K_{k+1}$.

\bigskip

\noindent{Claim III.}~~{\it If $\hat{G}_2\subset \hat{G}$ is a finitely generated subgroup then the functions in $\cup_{\phi\in \hat{G}_2}Q_\phi$ topologically generate a $k$-step character preserving nilspace factor.}

\bigskip

The group $\hat{G}_2$ is the direct product of finitely many cyclic groups $\hat{C}_1,\hat{C}_2,\dots,\hat{C}_n$. Let $\phi_i$ be a generator of $\hat{C}_i$ for $1\leq i\leq n$. 
If $\phi_i$ has finite order $d$ for some $i$ then by the previous claim we can assume (by abusing the notation) that $\rho_{\phi_i}^d=1$.
The dual group $G_2$ of $\hat{G}_2$ is embedded as a subgroup in $(\mathbb{C}^*)^n$ as $\oplus_{i=1}^n C_i$ where $C_i$ is the unit circle if $\phi_i$ has infinite order and $C_i$ is the group of $d$-th roots of unity if $\phi_i$ has order $d$. In this embedding let $\beta_i$ denote the $i$-th coordinate. 
The functions $\{\beta_i\}_{i=1}^n$ form a generating system for the dual group of $G_2$.

Now let $\rho:\bA\rightarrow G_2$ be defined by $$\rho(x)=(\rho_{\phi_1}(x),\rho_{\phi_2}(x),\dots,\rho_{\phi_n}(x)).$$
We can define $(\rho)^\diamond$ using component wise operations. We have that there is a measurable cocycle $\kappa:C^{k+1}(N_{k-1})\rightarrow G_2$ such that $\rho$ is the composition of $\Psi_{k-1}$ and $\kappa$.
The cocycle $\kappa$ defines a measurable nilspace extension $N$ of $N_{k-1}$ by the group $G_2$. 
Our goal is to construct a continuous morphism $\Psi$ from $\bA$ to $N$. Then we will see that $\Psi$ is character preserving.

We use proposition \ref{meascont} to describe the structure of $N$. The elements of $N$ of the fiber above $x\in N_{k-1}$ are represented by shifts of the function $\kappa$ restricted to $C^{k+1}_x(N_{k-1})$ by elements from $G_2$. Using the multiplicative notation, this means that we choose an element $g\in G_2$ and then we multiply the restriction of $\kappa$ to $C^{k+1}_x(N_{k-1})$ by $g$ at every point. 
We define $\Psi:\bA\rightarrow N$ as the function which maps $x\in\bA$ to the function which is the restriction of $\kappa\rho(x)^{-1}$ to $C^{k+1}_y(N_{k-1})$ where $y=\Psi_{k-1}(x)$. In other words $\Psi(x)$ is the image of $(\rho)^\diamond_x$ under $\Psi_{k-1}$.

The fact that $\Psi$ is a morphism follows directly from the construction of the cubic structure on $N$.
For an element $f:C_x^{k+1}(N_{k-1})\rightarrow G_2$ which is $g$ times the restriction of $\kappa$ to $C_x^{k+1}(N_{k-1})$ we define $\beta_i'(f)$ as $\beta_i(g)$. It is clear from the definitions that $\Psi$ composed with $\beta_i'$ is equal to $\rho_{\phi_i}^{-1}$. This shows that $\Psi$ is character preserving. The see continuity we have to show that $\Psi$ composed with the functions in lemma \ref{topgen} are all continuous on $\bA$.
On the other hand by the rooted measure preserving property of $\Psi_{k-1}$ these functions are exactly the functions in $\cup_{\phi\in \hat{G}_2}Q_\phi$.
Since corner convolutions are continuous on $\bA$ we obtain the continuity of $\Psi$. Theorem \ref{charpres} shows that $\Psi$ is a nilspace factor of $\bA$. This finishes the proof of the claim.

\bigskip

The general case follows from the fact that $\hat{G}$ is the direct limit of its finitely generated subgroups. The construction in claim III. shows that if $\hat{G_2}\subset\hat{G_3}$ are two finitely generated subgroups in $\hat{G}$ then there is a continuous morphism from the nilspace factor corresponding to $\hat{G_3}$ to the factor corresponding to $\hat{G_2}$. It shows that all the characters $\chi_\phi$ are measurable in the inverse limit of these factors. This inverse limit is again character preserving so it is a factor. This completes the proof.

\bigskip

\noindent{\it Proof of theorem \ref{ultdecomp}:~} Let $f_s=\mathbb{E}(f|\mathcal{F}_k)$ and $f_r=f-f_s$. We have that $\|f_r\|_{U_{k+1}}=0$. By lemma \ref{sepnil} the $\sigma$-algebra generated by $f_s$ is contained in a degree-$k$ nil-$\sigma$-algebra $\mathcal{B}$. Then by theorem \ref{nilnil} we have that $f_s$ is measurable in a $k$-step character preserving nilspace factor of $\bA$.
 
\bigskip

\subsection{Regularization, inverse theorem and special families of groups} 

\bigskip

We prove theorem \ref{reglem}, theorem \ref{invthem}, theorem \ref{restreg} and theorem \ref{restinv}.

\noindent{\it Proof of theorem \ref{reglem}}~We proceed by contradiction. Let us fix $k$ and $F$. Assume that the statement fails for some $\epsilon>0$. 
This means that there is a sequence of measurable functions $\{f_i\}_{i=1}^\infty$ on the compact abelian groups $\{A_i\}_{i=1}^\infty$ with $|f_i|\leq 1$ such that $f_i$ does not satisfy the statement with $\epsilon$ and $n=i$. Let $\omega$ be a fixed non-principal ultra filter and $\bA=\prod_\omega A_i$.
We denote by $f$ the ultra limit of $\{f_i\}_{i=1}^\infty$.
By theorem \ref{ultdecomp} we have that $f=f_s+f_r$ where $\|f_r\|_{U_{k+1}}=0$ and $f_s=\Psi\circ g$ for some character preserving nilspace factor $\Psi:\bA\rightarrow N$ and measurable function $g:N\rightarrow\mathbb{C}$.
Now we use the theorem proved in \cite{NP} which says that $N$ is an inverse limit of finite dimensional nilspaces $\{N_i\}_{i=1}^\infty$.
Let $\mathcal{N}_i$ denote the $\sigma$-algebra generated by the projection to $N_i$. Then we have that 
$g=\lim_{i\to\infty}\mathbb{E}(g|\mathcal{N}_i)$ in $L^1$.
It follows that there is an index $j$ such that
$g_j=\mathbb{E}(g|\mathcal{N}_j)$ satisfies $\|g-g_j\|_1\leq\epsilon/3$.
Furthermore there is a Lipschitz function $h:N_j\rightarrow\mathbb{C}$ with $|h|\leq 1$ and Lipschitz constant $c$ such that $\|g_j-h\|_1\leq\epsilon/3$.

Since the factor $\Psi$ and the projection to $N_j$ is measure preserving we have that $q=\Psi\circ h$ satisfies that $\|q-f_s\|\leq 2\epsilon/3$.
Let $f_e=f-f_r-q$. 
The function $\Psi$ is a continuous function so there is a sequence of continuous functions $\{\Psi':A_i\rightarrow N_j\}_{i=1}^\infty$ such that $\lim_{\omega}\Psi'_i=\Psi$. It is easy to see that $\Psi'_i$ is an approximate morphism with error tending to $0$.
It follows from \cite{NS} that it can be corrected to a morphism $\Psi_i$ (if $i$ is sufficiently big) such that the maximum point wise distance of $\Psi_i$ and $\Psi_i'$ goes to $0$. As a consequence we have that $\lim_\omega\Psi_i=\Psi$.

Let $f^i_s=\Psi_i\circ h$, and let $f^i_r$ be a sequence of measurable functions with $\lim_\omega f^i_r=f_r$.
We set $f^i_e=f_i-f^i_s-f^i_r$. It is clear that $\lim_\omega f^i_s=q$ and $\lim_\omega f^i_e=f_e$.
We also have that $\lim_\omega\|f^i_r\|_{U_{k+1}}=\|f_r\|_{U_{k+1}}=0$ and $\lim_\omega (f^i_r,f^i_e+f^i_s)=(f_r,f_e+q)=0$.
Let $m$ be the maximum of the complexity of $N_i$ and $c$.
There is an index set $S$ in $\omega$ such that 
\begin{enumerate}
\item $\|f^i_r\|_{U_{k+1}}\leq F(\epsilon,m)$,
\item $\|f^i_e\|_1\leq \epsilon$,
\item $|(f^i_r,f^i_s+f^i_e)|\leq F(\epsilon,m)$,
\item $\Psi_i$ is at most $F(\epsilon,m)$ balanced,
\item $|\|f^i_s+f^i_e\|_{U_{k+1}}-\|f_i\|_{U_{k+1}}|\leq F(\epsilon,m)$,
\end{enumerate}
hold simultaneously on $S$.
Note that $\Psi$ itself is $0$ balanced.
This is a contradiction.

\bigskip

\noindent{\it Proof of theorem \ref{restreg}}~In the proof of \ref{reglem} the nilspace $N_j$ that we construct is a character preserving factor of $\bA$. This means that the $i$-th structure group of $N_j$ is embedded into $\hat{\bA}_i$. This shows that $N_j$ is a $\mathfrak{A}$-nilspace.  

\bigskip

\noindent{\it Proof of theorem \ref{invthem} and theorem \ref{restinv}}~It is clear that if we apply theorem \ref{reglem} with $\epsilon_2>0$ and function $F(a,b)=a/b$ then in the decomposition $f=f_s+f_e+f_r$ the scalar product $(f,f_s)$ is arbitrarily close to $(f_s,f_s)$ and $\|f_s\|_{U_{k+1}}$ is arbitrarily close to $\|f\|_{U_{k+1}}$ if $\epsilon_2$ is small enough (depending only on $\epsilon$). This means by corollary \ref{l2becs} that $(f_s,f_s)\geq2\epsilon^{2^k}/3$ holds if $\epsilon_2$ is small and also $(f,f_s)\geq\epsilon^{2^k}/2$ holds simultaneously.

The inverse theorem is special families follows in the same way from theorem \ref{restreg}.

\subsection{Limit objects for convergent function sequences}

First we first focus on the proof of theorem \ref{simplim}.

\bigskip

Let $\{f_i:A_i\rightarrow\mathbb{C}\}$ be a sequence of functions with $|f_i|\leq r$.
Let $\bA$ be the ultra product of $\{A_i\}_{i=1}^\infty$ and $f$ be the ultra limit of $\{f_i\}_{i=1}^\infty$.
We have that $M(f)=\lim_\omega M(f_i)=\lim M(f_i)$ for every moment $M$.

Let $g$ denote the projection of $f$ to the $\sigma$-algebra $\mathcal{F}$ generated by $\{\mathcal{F}_i\}_{i=1}^\infty$ and let $g_i=\mathbb{E}(g|\mathcal{F}_i)$. Since $\mathcal{F}_i$ is an increasing sequence of $\sigma$-algebras we have that $\lim_{i\to\infty}g_i=g$ in $L^1$.
By corollary \ref{cornineqcor} we have that $M(g)=M(f)$ for every moment $M$.

We define a sequence $\mathcal{C}_i$ of nil-$\sigma$-algebras recursively. Let $\mathcal{C}_1$ be a nil-$\sigma$-algebra in which $g_1$ is measurable. Let $\mathcal{C}_i\subset\mathcal{F}_i$ be a nil-$\sigma$-algebra containing $\mathcal{C}_{i-1}$ in which $g_i$ is measurable.
Finally define $\mathcal{C}$ as the $\sigma$-algebra generated by the system $\{\mathcal{C}_i\}_{i=1}^\infty$.
It is clear that $\mathcal{B}_i=\mathcal{C}\cap\mathcal{F}_i$ is a nil $\sigma$-algebra.
As in theorem \ref{nilnil} we can create a sequence $\Psi_i$ of nilspace factors generating the $\sigma$-algebra $\mathcal{B}_i$ in a way that these factors form an inverse system.
Let $\Psi:\bA\rightarrow N$ be the inverse limit of these factors. 

Since all the functions $g_i$ are measurable in the $\sigma$-algebra generated by $\Psi$ the function $g$ is also measurable in it. This means that there is a function $h:N\rightarrow\mathbb{C}$ such that $\Psi\circ h=g$.
Using the rooted measure preserving property of the factors $\Psi_i$ we have that $M(f)=M(g)=M(g_i)=M(\mathbb{E}(h|\Psi_i))=M(h)$ for every simple moment of degree $i$.
This completes the proof. 

\bigskip

The proof of theorem \ref{genlim} goes in a very similar way. The only difference is that we project all the functions $f^a\overline{f^b}$ with $a,b\in\mathbb{N}$ to $\mathcal{F}$. The resulting function system $g^{a,b}$ at almost every point $x$ describes the moments of a probability distribution on the complex disc of radius $r$.

\vskip 0.2in

\noindent
Bal\'azs Szegedy
\noindent
University of Toronto, Department of Mathematics,
\noindent
St George St. 40, Toronto, ON, M5R 2E4, Canada

\end{document}